\documentclass[preprint,12pt,3p]{elsarticle}
\usepackage{lmodern}
\usepackage{amssymb, amsmath, amsfonts,amssymb}
\usepackage{amsthm}
\usepackage{bbm}
\usepackage{enumitem} 
\usepackage{hyperref}
\numberwithin{equation}{section}

\theoremstyle{plain}
\newtheorem{theorem}{Theorem}[section]
\newtheorem{lemma}{Lemma}[section]
\newtheorem{proposition}{Proposition}[section]

\theoremstyle{remark}
\newtheorem{remark}{Remark}[section]
\newtheorem{example}{Example}[section]
\usepackage{lineno}

\newcommand{\n}{\mathbb{N}}
\newcommand {\<}{\left\langle}  
\renewcommand {\>}{\right\rangle}  
\newcommand {\norma}[1]{\left\|#1\right\|}
\newcommand {\lnorma}[2]{\left\|#2\right\|_{\mathcal{L}^2(\mathbb{P}_{#1})}}
\newcommand {\Lnorma}[3]{\left\|#3\right\|_{\mathcal{L}^2_{#1}(\lambda\,\otimes\, \mathbb{P}_{#2})}}
\newcommand{\wi}{\widehat{\mathcal{I}}}
\renewcommand{\i}{\mathcal{I}}
\newcommand{\I}[3]{\mathcal{I}^{(#1)}_{g,\,#2,\,#3}}
\renewcommand{\S}[3]{\mathcal{S}^{(#1)}_{g,\,#2,\,#3}}
\newcommand{\wH}{\widehat{\mathcal{H}}_{n,N}}
\renewcommand{\H}{\mathcal{H}_{n,N}}
\newcommand{\s}[2]{t_{#2}(#1)}
\newcommand{\hs}[2]{\hat{t}_{#2}(#1)}
\newcommand{\ew}{\mathbb{E}}
\newcommand{\pr}{\mathbb{P}}
\newcommand{\lip}{\operatorname{Lip}}
\newcommand{\cl}{\operatorname{cl}}

\journal{Journal of Differential Equations}

\begin{document}
\begin{frontmatter}

\title{Law of the Iterated Logarithm for Markov Semigroups with Exponential Mixing in the Wasserstein Distance}

\author[us]{Dawid Czapla\corref{cor}}
\ead{dawid.czapla@us.edu.pl}
\cortext[cor]{Corresponding author}

\author[lu]{Sander C. Hille}
\ead{shille@math.leidenuniv.nl}

\author[us]{Katarzyna Horbacz}
\ead{katarzyna.horbacz@us.edu.pl}

\author[us]{Hanna Wojew\'odka-\'Sci\k{a}\.zko}
\ead{hanna.wojewodka@us.edu.pl}

\affiliation[us]{organization={Institute of Mathematics, University of Silesia in Katowice},
           addressline={Bankowa 14},
           city={Katowice},
           postcode={40-007},
           country={Poland}}

\affiliation[lu]{organization={Mathematical Institute, Leiden University},
            addressline={Einsteinweg 55}, 
            city={Leiden},
            postcode={2333 CC}, 
            country={the Netherlands}}
            
\begin{abstract}
In this paper, we establish the law of the iterated logarithm for a wide class of non-stationary, continuous-time Markov processes evolving on Polish spaces. Specifically, our result applies to certain additive functionals of processes governed by stochastically continuous Markov-Feller semigroups that exhibit exponential mixing and non-expansiveness in the Wasserstein distance, provided that a suitable moment condition involving the initial distribution is satisfied. Furthermore, we outline the application of this result to a Markov process arising as the solution of an infinite-dimensional stochastic differential equation with dissipative drift and additive noise.
\end{abstract}


\end{frontmatter}
\section{Introduction}

Typically, the main objective in studying Markov processes is to understand their asymptotic behavior. One may be interested, for instance, in ergodicity, the rate at which the law of a~process converges to its stationary distribution, or the validity of classical limit theorems.
In this paper, we focus on the latter, proposing a tractable criterion ensuring the law of the iterated logarithm (LIL) for some additive functionals of Markov processes, formulated within the framework of transition semigroups acting on measures -- a setting that appears particularly well-suited for describing the evolution of process distributions.

It should be emphasized that the main result of this article (Theorem \ref{thm:main}) pertains to processes with a general, not necessarily locally compact state space -- specifically, a~Polish metric space, further denoted by $(X, \rho)$. As is well-known, the lack of local compactness rules out many classical techniques for analyzing the asymptotic behavior of Markov processes, such as those developed by Meyn and Tweedie (see, e.g.,~\cite{meyn_tweedie, mt93, mt}) or their various adaptations (see e.g.,~\hbox{\cite{benaim1, costa}}). These methods typically rely on Harris recurrence (often ensured via H\"ormander-type bracket conditions; cf. \cite{benaim1}), aperiodicity, and on certain criteria involving drift conditions towards a~petite set (see \hbox{\cite[Theorem 17.5.3]{mt93}}). Consequently, they are essentially restricted to $\psi$-irreducible processes. In non-locally compact settings, however, the $\psi$-irreducibility usually requires rather highly restrictive assumptions. For example, it is sometimes achieved (or effectively substituted) by combining open set (topological) irreducibility together with the strong Feller property (which is enabled by \hbox{\cite[Proposition 6.1.5]{mt93}}; see, e.g., \cite{b:priola1} or \hbox{\cite[Theorem 1.2]{b:hairer_sf}}, where the latter is proved in \cite{b:hairer_sf2}) or with equivalence of transition probabilities at a fixed time (see, e.g.,~\cite{b:priola2}).
For a~somewhat more detailed discussion of these issues, see, e.g., the introduction of \cite{pierwsza_praca} and the abstract of~\cite{hairer_mattingly}.

Before delving into the motivation behind our main result, let us first give a short overview of the LIL itself. 
First of all, the LIL can be seen as a~refinement of the strong law of large numbers, sharpening the convergence rate from $\mathcal{O}(t)$ to $\mathcal{O}(\ln\ln t)$. Furthermore, it may also be regarded as a complement to the central limit theorem (CLT), offering a~pathwise counterpart to its distributional statement. While the CLT captures the average fluctuations around the mean by establishing convergence in distribution to a normal law, the LIL provides a~more precise almost sure description of the extremal behavior, identifying the exact envelope within which the normalized trajectories oscillate. More specifically, it determines the almost sure lower and upper limits of appropriately scaled sample paths of a~stochastic process. To put it formally, a~real-valued stochastic process $\{X_t\}_{t \geq 0}$ satisfies the LIL whenever
\begin{equation}\label{lil-intro}
\liminf\limits_{t \to \infty}\frac{X_t}{\sqrt{2\sigma^2\,t\ln\ln t}} = -1 \quad \text{and} \quad \limsup\limits_{t \to \infty}\frac{X_t}{\sqrt{2\sigma^2\,t\ln\ln t}} = 1 \quad \text{a.s.,}
\end{equation}
where $\sigma>0$ can be viewed as the asymptotic variance of this process.

Originally formulated for independent and identically distributed random variables, both the CLT and the LIL were later extended to martingales (see, e.g., \cite{hall_heyde,b:heyde_scott, mt93}). These extensions laid the groundwork for the development of various adaptations of these limit theorems to the setting of Markov processes. Initial generalizations focused exclusively on stationary processes. Within the realm of continuous-time dynamics, one may refer, for instance, to the CLT and LIL for stationary ergodic Markov processes established in \cite{b:bhattacharya}, or to the CLT for stationary Markov processes with normal generators, proved in \cite{holzmann05}, which extends earlier results of \cite{gordin_lifsic81} (see also \cite{klo}).

Over time, a few versions of limit theorems for general, non-stationary, continuous-time Markov processes on Polish spaces have been developed. The most general result of this kind to date appears to be \cite[Theorem 2.1]{b:komorowski}, which establishes the CLT for Markov-Feller continuous-time processes under the so-called exponential mixing in the Wasserstein distance~$d_W$. Specifically, letting $\{P_t\}_{t\geq 0}$ denote the Markov semigroup under consideration (acting on probability measures), and assuming that each $P_t$ preserves the finiteness of measures' first moments, this property can be defined as follows: there exist constants $c,\gamma>0$ such that, for all $t\geq 0$ and any two probability measures $\nu_1,\nu_2$ with finite first moments, we have
\begin{equation}\label{e:war_kom}
d_W\left(\nu_1 P_t, \nu_2 P_t\right)\leq ce^{-\gamma t} d_W(\nu_1,\nu_2).
\end{equation}
Naturally, in addition to this mixing property, certain assumptions guaranteeing the uniform integrability of $\rho(x_0,\cdot)^{2+\delta}$ with respect to $\mu P_t$, $t\geq 0$ -- where $x_0\in X$, $\delta>0$, and $\mu$ is the initial distribution of an associated Markov process -- are required too. A discrete-time counterpart of \cite[Theorem 2.1]{b:komorowski} is given in \cite[Theorem 5.1]{gulgowski}. What is both noteworthy and somewhat surprising about the latter is that order-$2$ uniform integrability is sufficient instead of the aforementioned, more typical `order-strictly-higher-than-2' condition, used in~\cite{b:komorowski} and also employed in the present paper (see Assumption \ref{cnd:a4}).

With regard to the LIL, a criterion analogous in spirit to \cite[Theorem 2.1]{b:komorowski}, but formulated for discrete-time Markov chains, was established in \cite{b:bolt_majewski}. To the best of our knowledge, however, no such result addressing the LIL in the continuous-time setting has been obtained so far. This gap in the literature actually constitutes the primary motivation for the present work. Besides \cite{b:komorowski} and \cite{b:bolt_majewski}, our work is also inspired by LIL-type results for some particular classes of Markov processes, including, e.g.,\cite{b:komorowski_LIL}.

In our framework we are nevertheless compelled to adopt a slightly stronger version of condition \eqref{e:war_kom} by requiring $c = 1$, albeit in a somewhat weaker formulation regarding its domain of validity, here restricted to Dirac measures. Specifically, we assume that there exists $\gamma > 0$ such that for every $t \geq 0$ and any $x, y \in X$,
\begin{equation}\label{exp_mixing}
d_W\left(\delta_x P_t, \delta_y P_t\right) \leq e^{-\gamma t} \rho(x, y).
\end{equation}
Although we are aware of the limitations this reinforcement imposes on the applicability of our result, all indications suggest that it is unavoidable within our approach to proving Theorem~\ref{thm:main}. In short, the necessity of assuming $c=1$ arises from the specific structure of the proof of Lemma~\ref{lem:5}, which forms a crucial step in the overall argument leading to the main result (a~bit more detailed justification is provided in Section~\ref{sec:2}). 

At the same time, it should be stressed that, despite this stronger form of exponential mixing, our LIL criterion remains verifiable in a number of significant models. Notably, it \hbox{applies} to the solution of an Itô stochastic differential equation (SDE) with dissipative drift and additive noise, studied in \cite[§6.1]{b:komorowski} in connection with the CLT. For the convenience of the reader, this example is briefly revisited at the end of the paper. Moreover, hypothesis~\eqref{exp_mixing} is also met in other relevant models, such as those discussed in \hbox{\cite[§6.1]{b:komorowski}} and \hbox{\cite[§4.1, §5.2]{b:cloez_hairer}}, thereby opening the way to establishing the LIL in these settings. Consequently, our result appears to at least partially bridge a previously unexplored gap in the general theory of limit theorems for continuous-time Markov processes on Polish state spaces.

Finally, let us pay some attention to the proof of the main result. The overall strategy proceeds in three main steps: reducing the problem to the discrete-time setting, applying a martingale-based decomposition to the Markov process under consideration, and establishing the LIL for the resulting martingale along the integers. To accomplish the latter, we employ the Heyde--Scott criterion (\cite[Theorem 1]{b:heyde_scott}), in a manner similar to \cite{b:bolt_majewski}. In the case where the process starts from its stationary distribution, the core conditions of this criterion (see \eqref{cnd:b1} and \eqref{cnd:b2}) can be derived relatively easily by applying the Birkhoff ergodic theorem (cf. Lemma~\ref{lem:4}). The main difficulty lies in verifying that they remain valid in the non-stationary setting. The cornerstone of our approach in this regard is the above-mentioned Lemma~\ref{lem:5}, which constitutes the most technically demanding part of the proof and, arguably, the principal contribution of the paper. More precisely, Lemma~\ref{lem:5} establishes the continuity of certain functions of the form $x \mapsto \mathbb{E}_x[Y]$, where $Y$ is a random variable depending on $\limsup$ or $\liminf$ of the squared martingale increments. The key idea behind this technique originates from \cite{b:bolt_majewski} (see also \cite{b:czapla-strassen}), where it was developed in the context of martingale increments associated with a~discrete-time Markov chain. Lemma~\ref{lem:5} essentially demonstrates that this method remains applicable -- though under a~more restrictive form of exponential mixing condition -- when the martingale increments arises from a~continuous time process.

The paper is organized as follows. Section \ref{sec:1} collects all the necessary definitions and notation, including basic concepts from the theory of Markov processes and their transition semigroups, as well as the formulas for the distances in the space of probability measures used throughout the paper. Section~\ref{sec:2} presents the formulation of the main theorem \hbox{(Theorem~\ref{thm:main})}, preceded by an introduction and discussion of the assumptions, and by a version of the exponential ergodicity result from \cite{b:czapla-CLT-cont}, adapted to the present setting \hbox{(Proposition \ref{prop:1})}. 
Sections~\ref{sec:3} and~\ref{sec:4} are devoted to the proof of Theorem \ref{thm:main}. Section \ref{sec:3} outlines the general strategy, while Section~\ref{sec:4} contains the detailed proof of Lemma \ref{lem:5}, together with all the necessary auxiliary results. Section~\ref{sec:5} discusses representations of the asymptotic variance involved in the theorem. Finally, Section~\ref{sec:6} briefly illustrates an application of our main result to the above-mentioned SDE with dissipative drift and additive noise. 

\section{Preliminaries} \label{sec:1}
Throughout the paper, we shall work on a complete separable metric space $(X,\rho)$, endowed with its Borel $\sigma$-field $\mathcal{B}(X)$. As usual, the symbol $B(x,r)$ will represent the open ball centered at $x$ with radius~\hbox{$r>0$} in this space.

By $BM(X)$ we will denote the space of all real-valued, Borel measurable, bounded functions on $X$, equipped with the supremum norm $\norma{\cdot}_{\infty}$, and we will write $BC(X)$ to denote the subspace of $BM(X)$ consisting of all continuous functions. Further, $L(X)$ will stand for the family of all real-valued Lipschitz continuous functions on~$X$, i.e., those \hbox{$f:X\to\mathbb{R}$} for which
$$\lip f:=\sup\left\{\frac{|f(x)-f(y)|}{\rho(x,y)}:\; x,y\in X,\; x\neq y\right\}<\infty,$$
while $BL(X)$ will represent the space of all bounded functions in $L(X)$, endowed with the norm~$\norma{\cdot}_{BL}$, given by
$$\norma{f}_{BL}:=\norma{f}_{\infty}+\lip f.$$

Moreover, we let $\mathcal{M}_1(X)$ denote the family of all Borel probability measures on $X$, and, for every $r>0$, we define $\mathcal{M}_{1,r}(X)$ as the subfamily of $\mathcal{M}_1(X)$ consisting of all measures with finite $r$-th moment, i.e., those $\mu\in\mathcal{M}_1(X)$ that satisfy $\int_X \rho(x_0,x)^r(x)\,\mu(dx)<\infty$, with some $x_0\in X$ (clearly, this definition is independent of the choice of $x_0$). The symbol $\delta_x$ will stand for the Dirac measure at $x$, i.e., $\delta_x(A):=\mathbbm{1}_A(x)$ for $A\in \mathcal{B}(X)$. Additionally, given any Borel measurable function $f:X\to\mathbb{R}$ and any finite signed Borel measure~$\mu$ on $X$, we shall use the notation 
 $$\<f,\mu\>:=\int_X f(x)\,\mu(dx),$$
whenever the integral on the right-hand is well-defined.

Let us recall that a sequence $\{\mu_n\}_{n\in\n}\subset \mathcal{M}_1(X)$ is called \emph{weakly convergent} to a measure $\mu\in\mathcal{M}_1(X)$ (which is denoted as $\mu_n\stackrel{w}{\to}\mu$) if $\lim_{n\to\infty} \<f,\mu_n\>=\<f,\mu\>$ for every $f\in BC(X)$. It is well known (see \cite[Theorems 6 and 8]{b:dudley}) that the weak convergence of Borel probability measures on $X$ can be metrized by the so-called \emph{bounded Lipschitz distance} (also referred to as Dudley's or Fortet--Mourier's metric), given by
$$d_{BL}(\mu,\nu):=\sup\left\{|\<f,\mu-\nu\>|:\; f\in BL(X),\; \norma{f}_{BL}\leq 1\right\}\quad\text{for}\quad \mu,\nu\in\mathcal{M}_1(X),$$
and that the space $(\mathcal{M}_1(X),d_{BL})$ is complete (see, e.g., \cite[Theorem 9]{b:dudley}). 

The metric $d_{BL}$ will only serve as an auxiliary tool in our study. The primary assumption of this paper, i.e., the exponential mixing property (see condition \ref{cnd:a3}), will be formulated in terms of the \emph{Wasserstein distance} (also known as the Kantorovich–Rubinstein or the Hutchinson metric; see \cite{b:hutchinson}), which is defined on $\mathcal{M}_{1,1}(X)$ as
$$d_W(\mu,\nu):=\sup\left\{|\<f,\mu-\nu\>|:\; f\in L(X),\; \lip f\leq 1\right\}\quad\text{for}\quad \mu,\nu\in\mathcal{M}_{1,1}(X).$$
Notably, this definition remains unchanged if the supremum is taken over $f\in BL(X)$ instead of $f\in L(X)$ (see, e.g., \cite[Lemma 3.1]{b:kapica}). Clearly, $d_W$ is stronger than $d_{BL}$. The main reason for requiring the aforementioned property under $d_W$ rather than $d_{BL}$ lies in the inequality
$$|\<f,\mu-\nu\>|\leq (\lip f)d_W(\mu,\nu)\quad\text{for all}\quad f\in L(X),\;\mu,\nu\in\mathcal{M}_{1,1}(X),$$
which will be crucial for establishing Lemma \ref{lem:5} (see Lemma \ref{lem-c:4}, used in its proof). By contrast, the analogous inequality involving $d_{BL}$ (for $f\in BL(X)$) holds with $\norma{g}_{BL}$ in place of $\lip f$, which appears insufficient for our approach.

In the following, we will review several key concepts and introduce the necessary notation related to Markov processes.

A function $P:X\times\mathcal{B}(X)\to [0,1]$ is called a \emph{stochastic kernel} on $X$ if, for each \hbox{$A\in\mathcal{B}(X)$}, $x\mapsto P(x,A)$ is Borel measurable, and, for every $x\in X$, $A\mapsto P(x,A)$ is a Borel probability measure.  The composition of any two such kernels, say $P$ and $Q$, is the kernel $PQ$ defined as
\begin{equation}
\label{e:composition}
PQ(x,A):=\<Q(\cdot, A),\, P(x,\cdot)\>\quad\text{for all}\quad x\in X,\; A\in\mathcal{B}(X).
\end{equation}

Given a stochastic kernel $P$ (on $X$), for any measure $\mu\in\mathcal{M}_1(X)$ and any bounded below Borel measureable function  $f:X\to\mathbb{R}$, we shall write $\mu P$ to denote the (Borel probability) measure defined as
$$\mu P(A):=\<P(\cdot,A),\,\mu\>\quad\text{for}\quad A\in\mathcal{B}(X),$$
and $Pf$ for the (Borel measurable) function acting from $X$ to $[0,\infty]$, given by
$$Pf(x):=\<f,\,P(x,\cdot)\>\quad\text{for}\quad x\in X.$$
Obviously, $\<f,\mu P\>=\<Pf,\mu\>$. The map \hbox{$(\cdot)P:\mathcal{M}_1(X)\to \mathcal{M}_1(X)$} is usually referred to as a~\emph{Markov operator}, while \hbox{$P(\cdot):BL(X)\to BL(X)$} is said to be its dual operator.

Let $\mathbb{T}\neq \emptyset$ be an additive submonoid of $(\mathbb{R},+,0)$. A family $\{P_t\}_{t\in\mathbb{T}}$ of stochastic kernels (on~$X$) is called a~Markov semigroup whenever the map $\mathbb{T}\times X\ni (t,x)\mapsto P_t(x,A)$ is \mbox{$\mathcal{B}(\mathbb{T})\otimes\mathcal{B}(X)/\mathcal{B}(\mathbb{R})$} - measurable for each $A\in\mathcal{B}(X)$, $P_0(x,\cdot)=\delta_x$ for every $x\in X$, and $P_{s+t}=P_sP_t$ (in the sense of~\eqref{e:composition}) for any~$s,t\in\mathbb{T}$. Clearly, the corresponding families of Markov operators and their duals then enjoy the latter property under the usual composition of functions.

A measure $\mu_*\in \mathcal{M}_1(X)$ is said to be \emph{invariant} (or \emph{stationary}) for a Markov semigroup~$\{P_t\}_{t\in\mathbb{T}}$ if $\mu_* P_t=\mu_*$ for all $t\in\mathbb{T}$.

Now, suppose we are given a Markov semigroup $\{P_t\}_{t\in\mathbb{T}}$, and let $(\Omega, \mathcal{F})$ be a measurable space, equipped with a set $\{\pr_x\}_{x\in X}$ of probability measures on it and a filtration \hbox{$\{\mathcal{F}_t\}_{t\in\mathbb{T}}\subset\mathcal{F}$}. Moreover, let $\{\theta_t\}_{t\in\mathbb{T}}$ be a family of mappings from $\Omega$ into itself. Following \hbox{\cite[Ch. I, §3]{b:blumenthal}}, we will understand a~time-homogeneous Markov process with state space $X$ and transition semigroup $\{P_t\}_{t\in\mathbb{T}}$ as the system $$\Phi:=\left(\Omega,\mathcal{F},\{\mathcal{F}_t\}_{t\in\mathbb{T}},\{\Phi_t\}_{t\in\mathbb{T}},\{\theta_t\}_{t\in\mathbb{T}},\{\pr_x\}_{x\in X}\right),$$ provided the following conditions hold:
\begin{gather}
\label{m:1}  \Phi_t:\Omega\to X\;\; \text{is}\;\;\mathcal{F}_t/\mathcal{B}(X)\text{ - measurable\; for every} \quad t\in\mathbb{T},\\
\label{m:2} \Phi_{s+t}=\Phi_s\circ \theta_t \quad\text{for any}\quad s,t\in\mathbb{T},\\
\label{m:3}\pr_x(\Phi_0=x)=1\quad \text{for every}\quad x\in X,\\
\pr_x(\Phi_{t+s}\in A\,|\,\mathcal{F}_s)=\pr_{\Phi_s}(\Phi_t\in A)\;\; \pr_x\text{-a.s.} \quad\text{for any} \quad x\in X, 
\label{m:4}\;A\in\mathcal{B}(X),\; s,t\in\mathbb{T},\\
\label{m:5}\pr_x(\Phi_t\in A)=P_t(x,A) \quad\text{for any}\quad  x\in X,\; A\in\mathcal{B}(X), \; t\in\mathbb{T}.
\end{gather}
If, additionally, for every $t\in\mathbb{T}$, the mapping $(\mathbb{T}\cap[0,t])\times \Omega \ni (s,\omega) \mapsto \Phi_s(\omega)\in X$ is~\hbox{$\mathcal{B}(\mathbb{T}\cap[0,t])\otimes \mathcal{F}_t\,/\,\mathcal{B}(X)$} - measurable, then $\Phi$ is called \emph{progressively measurable}. It~is worth noting here that \eqref{m:1} and \eqref{m:2} imply, in particular, that, for each $t\geq 0$, the shift~$\theta_t$ is both $\mathcal{F}_{s+t}/\mathcal{F}_s$-measurable with any $s\geq 0$ and $\mathcal{F}/\mathcal{F}$-measurable. 

For a Markov process $\Phi$ defined as above, it is clear that conditions \eqref{m:1} and \eqref{m:4} remain valid if~$\mathcal{F}$ and $\mathcal{F}_t$, $t\in\mathbb{T}$, are replaced by 
\begin{equation}
\label{e:filt}
\mathcal{F}^{\Phi}_{\infty}:=\sigma(\{\Phi_s:\;s\in\mathbb{T}\})\quad\text{and}\quad\mathcal{F}_t^{\Phi}:=\sigma(\{\Phi_s:\;s\in\mathbb{T}\cap [0,t]\}),\;\;t\in\mathbb{T},
\end{equation}
respectively. Furthermore, by repeatedly using the Markov property (i.e., \eqref{m:4} and \eqref{m:5}), it is easy to check that, for any $f\in BM(X^n)$, any $0<s_1<\ldots<s_n$ in $\mathbb{T}$, and every $x\in X$,
\begin{align}
\label{e:Ex}
\begin{split}
&\ew_x\left[f\left (\Phi_{s_1},\ldots,\Phi_{s_n}\right) \right]\\
&=\int_X\int_X\ldots\int_X f(x_1,\ldots,x_n)\,P_{s_n-s_{n-1}}(x_{n-1},dx_n)\ldots P_{s_2-s_1}(x_1,dx_2)P_{s_1}(x,dx_1),
\end{split}
\end{align}
where $\ew_x$ denotes the expectation with respect to $\pr_x$. This implies, in particular,  that the map $X\ni x\mapsto \ew_x[Y]$ is Borel measurable for each bounded $\mathcal{F}_{\infty}^{\Phi}$-measurable random variable~$Y$. For every $\nu\in\mathcal{M}_1(X)$, we can therefore define $\pr_{\nu}$ as
$$\pr_{\nu}(F):=\int_X \pr_x(F)\,\nu(dx)\quad \text{for}\quad F\in\mathcal{F}_{\infty}^{\Phi}.$$
Under this definition, \eqref{m:4} remains valid with $\pr_{\nu}$ in place of $\pr_x$. Putting it more generally, for any bounded $\mathcal{F}_{\infty}^{\Phi}$-measurable random variable $Y$ and any $s\geq 0$, we have
\begin{equation}
\label{e:mar_gen}
\ew_{\nu}\left[Y\circ \theta_s\,|\,\mathcal{F}_s\right]=\ew_{\Phi_s}[Y]\quad \pr_{\nu}\text{ - a.s.},
\end{equation}
where $\ew_{\nu}$ denotes the expectation operator under~$\pr_{\nu}$. Obviously, when the process $\{\Phi_t\}_{t\in\mathbb{T}}$ is considered on $(\Omega, \mathcal{F}_{\infty}^{\Phi},\pr_{\nu})$, then $\nu$ serves as its initial distribution, and $\nu P_t$ is the distribution of~$\Phi_t$ for every $t\in\mathbb{T}$.

Finally, let us introduce one more piece of notation. Given a probability space $(\Omega,\mathcal{F},\pr)$ and $r\in [1,\infty)$, as is customary, we will write $\mathcal{L}^r(\pr)$ for the of space all real-valued random variables on $(\Omega,\mathcal{F},\pr)$ with a finite $r$-th moment, equipped with the norm defined as
$$\norma{Y}_{\mathcal{L}^r(\pr)}:=\left(\ew\left[|Y|^r\right]\right)^{1/r}\quad\text{for}\quad Y\in\mathcal{L}^r(\pr).$$
Furthermore, in Section \ref{sec:4}, we will make use of the space $\mathcal{L}^2_{S,T}(\lambda\otimes\pr)$ (where $S<T$), consisting of all $\mathcal{B}([S,T])\otimes\mathcal{F}/\mathcal{B}(\mathbb{R})$-measurable functions $\eta:[S,T]\times\Omega\to\mathbb{R}$ such that
$$\Lnorma{S,T}{}{\eta}:=\left(\int_{[S,T]\otimes\Omega} \eta^2(t,\omega)\,(\lambda\otimes\pr)(dt\times d\omega)\right)^{1/2}<\infty,$$
where $\lambda$ stands for the $1$-dimensional Lebesgue measure.  When referring to this space, we will identify functions $\eta:[S,T]\times\Omega\to\mathbb{R}$ with processes \hbox{$\{\eta_t\}_{t\in [S,T]}$}.

\section{Assumptions and Formulation of the Main Result} \label{sec:2}
Let $\Phi:=\left(\Omega,\mathcal{F},\{\mathcal{F}_t\}_{t\geq 0},\{\Phi_t\}_{t\geq 0},\{\theta_t\}_{t\geq 0},\{\pr_x\}_{x\in X}\right)$ be a time-homogeneous, progressively measurable Markov process with state space $X$ and transition semigroup $\{P_t\}_{t\geq 0}$, where
$$\mathcal{F}=\mathcal{F}_{\infty}^{\Phi},\quad\text{and}\quad\mathcal{F}_t=\mathcal{F}_t^{\Phi}\quad\text{for}\quad t\geq 0,$$
are defined according to \eqref{e:filt}, with $\mathbb{T}=[0,\infty)$.
This assumption applies to the whole paper and will not be repeated afterwards.

In what follows, we also employ the following conditions:
\begin{enumerate}[label=\textnormal{(A\arabic*)}, leftmargin=*]
\item \label{cnd:a1} $\{P_t\}_{t\geq 0}$ is \emph{stochastically continuous}, i.e.,
$$\lim_{t\to 0^+} P_t f(x)=f(x)\quad \text{for all} \quad x\in X,\; f\in BC(X);$$
\item \label{cnd:a2} $\{P_t\}_{t\geq 0}$ is \emph{Feller}, i.e., $P_t f\in BC(X)$ for any $f\in BC(X)$ and $t\geq 0$;
\item \label{cnd:a3} We have $\nu P_t\in \mathcal{M}_{1,1}(X)$ for any $\nu\in\mathcal{M}_{1,1}(X)$, $t>0$, and there exists $\gamma\in (0,\infty)$ such that 
$$d_W(\delta_x P_t,\, \delta_y P_t)\leq e^{-\gamma t} \rho(x,y)\quad\text{for all}\quad t\geq 0,\; x,y\in X;$$
\item \label{cnd:a4} For a given measure $\mu\in\mathcal{M}_1(X)$, there exist $x_0\in X$ and $\zeta\in (2,\infty)$ such that
\begin{equation}
\label{e:a4}
\sup_{t\geq 0}\<V^{\zeta},\,\mu P_t\><\infty\quad\text{with}\quad V:=\rho(x_0,\cdot),
\end{equation}
which, in particular, ensures that $\mu\in\mathcal{M}_{1,\zeta}(X)$.
\end{enumerate}
Moreover, for the sake of further considerations, it will be convenient to adopt the notation
$$\quad c_r(\nu):=\sup_{t\geq 0}\<V^r,\,\nu P_t\> \quad\text{for}\quad \nu\in\mathcal{M}_1(X),\; r>0.$$

Regarding condition~\ref{cnd:a3}, as already mentioned in the introduction, our approach to proving the main result unfortunately does not permit the insertion of an arbitrary multiplier $c > 0$ on the right-hand side of the inequality stated therein, due to the structure of the proof of Lemma~\ref{lem:5}. To be more precise, the requirement that $P_t$ acts as a contraction on the set of Dirac measures for all $t > 0$ is critical for the proof of Lemma~\ref{lem-c:8}, which constitutes a key step in the argument leading to Lemma~\ref{lem:5}. This, in turn, is strictly related to the nature of the estimates involved in the proof of Lemma \ref{lem-c:4}, upon which Lemma~\ref{lem-c:8} relies. Nonetheless, as also previously noted, it is not difficult to identify models in which condition~\ref{cnd:a3} is satisfied.

\begin{remark}\label{rem:1} From condition \ref{cnd:a3} it follows that 
\begin{equation*}
d_W(\nu_1 P_t,\nu_2 P_t )\leq e^{-\gamma t}\<V,\, \nu_1+\nu_2\>\quad\text{for all}\quad t\geq 0,\; \nu_1,\nu_2\in\mathcal{M}_{1,1}(X).
\end{equation*}
Indeed, it suffices to observe that, for every function $f\in L(X)$ with $\lip f\leq 1$ and any measures \hbox{$\nu_1,\nu_2\in\mathcal{M}_{1,1}(X)$},
\begin{align*}
|\<f,\nu_1 P_t\>-\<f,\nu_2 P_t\>|&\leq \int_X \int_X |P_t f(x)-P_t f(y)|\,\nu_1(dx)\nu_2(dy) 
\\
&\leq \int_X \int_X d_W(\delta_x P_t, \,\delta_y P_t)\,\nu_1(dx)\nu_2(dy)\\
&\leq e^{-\gamma t} \int_X \int_X \rho(x,y)\,\nu_1(dx)\nu_2(dy)\\
&\leq e^{-\gamma t} \int_X \int_X (V(x)+V(y))\,\nu_1(dx)\nu_2(dy)=e^{-\gamma t}\<V,\, \nu_1+\nu_2\>.
\end{align*}
\end{remark}

\begin{remark}\label{rem:2} For every $\nu\in\mathcal{M}_1(X)$, if $c_{\zeta}(\nu)<\infty$, then $c_r(\nu)<\infty$ for all $r\in (0,\zeta]$. Consequently, in particular, condition \ref{cnd:a4} guarantees that $c_r(\mu)<\infty$ for all $r\in (0,\zeta]$. To see this, it is enough to note that H\"older's inequality (applied with exponents $p=\zeta/r$ and $q=p/(p-1)$ to functions $V$ and $\mathbbm{1}_X$, respectively) gives $$\<V^r, \nu\>\leq\<V^{\zeta},\nu\>^{r/\zeta},$$
which, in turn, leads to $c_r(\nu)\leq c_{\zeta}^{r/\zeta}(\nu)$. 
\end{remark}

\begin{remark}\label{rem:3}
Clearly, if there exist constants $a,b,\delta\geq 0$ such that 
\begin{equation}\label{e:Lap_drift}
P_tV^{\zeta}(x)\leq a e^{-\delta t}V^{\zeta}(x)+b \quad\text{for any}\quad t\geq 0,\; x\in X,
\end{equation}
then \ref{cnd:a4} holds for every $\mu\in\mathcal{M}_{1,\zeta}(X)$. When \eqref{e:Lap_drift} is satisfied with $\delta>0$, it is commonly referred to as a \emph{Lyapunov drift condition} for the semigroup $\{P_t\}_{t\geq 0}$ with Lyapunov function~$V^{\zeta}$.
\end{remark}

\begin{remark}
It is clear that hypotheses \ref{cnd:a3} and \ref{cnd:a4} yield that $\{P_t\}_{t\geq 0}$ enjoys the Lyapunov condition with $V$. Indeed, referring to Remark \ref{rem:1}, we see that for any $x\in X$ and~$t\geq 0$,
\begin{align*}
P_t V(x)&=P_t V(x) - \<V,\mu P_t\>+\<V,\mu P_t\>=\<V,\, \delta_x P_t - \mu P_t\>+\<V,\mu P_t\>\\
&\leq d_W(\delta_x P_t,\,\mu P_t)+\<V,\mu P_t\>\leq e^{-\gamma t}\<V,\delta_x+\mu\>+\<V,\mu P_t\>\\
&\leq e^{-\gamma t} V(x)+\<V,\mu\>+\<V,\mu P_t\>\leq e^{-\gamma t} V(x)+2c_1(\mu).
\end{align*}
One may therefore ask whether these two conditions might also imply the Lyapunov drift with~$V^{\zeta}$. In view of Remark \ref{rem:3}, this would mean that \ref{cnd:a3} and \ref{cnd:a4} actually guarantee that \eqref{e:a4} holds for every measure $\mu\in\mathcal{M}_{1,\zeta}(X)$, and thus the forthcoming Theorem \ref{thm:main} would be valid for all initial distributions in $\mathcal{M}_{1,\zeta}(X)$. Unfortunately, this is not the case, as demonstrated in the counterexample below.
\end{remark}

\begin{example}\label{example:1}
Fix arbitrary $\gamma>0$ and $\zeta>2$ (in fact, any $\zeta>1$ suffices here), and take $x_0=0$. We will construct a time-homogeneous, progressively measurable Markov process with state space $(\mathbb{R},\,|\cdot|)$ whose transition semigroup enjoys hypotheses \ref{cnd:a1}-\ref{cnd:a3}, while \ref{cnd:a4} (with $V(x)=|x|$) holds only for $\mu=\delta_0$, and in particular condition \eqref{e:Lap_drift} fails.

On a suitable probability space $(\Omega,\mathcal{F},\pr)$, consider a~Poisson process $\{N_t\}_{t\geq 0}$ with rate $1$ and a sequence $\{Y_n\}_{n\in\n}$ of mutually independent, identically distributed random variables, independent of $\{N_t\}_{t\geq 0}$, with the two-point law
$\pr(Y_1=K)=\pr(Y_1=0)=1/2$, where $K>0$ is large enough so that
\begin{equation}\label{e:large_M}
e^{\zeta K}-\zeta e^{K}>2\zeta\gamma.
\end{equation}
Further, introduce the compound Poisson process $\{X_t\}_{t\geq 0}$ given by
$$X_0:=0,\quad X_t=\sum_{k=1}^{N_t} Y_t\quad\text{for}\quad t>0,$$
and define 
\begin{equation}\label{def:contr_Markov}
\Phi_t:=\Phi_0 \exp(-\alpha t+X_t)\quad\text{for}\quad t>0,\quad \text{with}\quad \alpha:=\frac{1}{2}\left(e^K-1\right)+\gamma>0,
\end{equation}
where $\Phi_0$ is an $\mathbb{R}$-valued random variable that is independent of $\{X_t\}_{t\geq 0}$. Moreover, let $\mathbb{R}\times\mathcal{F}\ni (x,F)\mapsto\pr(F\,|\,\Phi_0=x)$ be a regular conditional distribution of~$\pr$ given~$\Phi_0$ (which exists as $(\mathbb{R},\mathcal{B}(\mathbb{R}))$ is a standard Borel space), and put $\pr_x=\pr(\cdot\,|\,\Phi_0=x)$ for $x\in\mathbb{R}$.

We will first show that, for every $x\in \mathbb{R}$, the process $\Phi$ defined in \eqref{def:contr_Markov} enjoys the Markov property under $\pr_x$ with respect to its natural filtration $\{\mathcal{F}_t\}_{t\geq 0}$. To this end, we begin by proving that $\Phi$ is a Markov process with respect to the natural filtration $\{\mathcal{H}_t\}_{t\geq 0}$ of~$\{X_t\}_{t\geq 0}$. Let~$A\in\mathcal{B}(\mathbb{R})$, $s,t\geq 0$, and set $X_t^{(s)}:=X_{t+s}-X_s$. Then
\begin{align}\label{e:phits}
\begin{split}
\Phi_{t+s}&=\Phi_0\exp\left(-\alpha(t+s)+X_{t+s} \right)=\Phi_0 \exp\left(-\alpha s +X_s \right) \exp\left(-\alpha t +X_{t+s}-X_s \right)\\
&=\Phi_s \exp\left(-\alpha t+X_t^{(s)}\right).
\end{split}
\end{align} 
Furthermore, due to the independent and stationary increments of the compound Poisson process, $X_t^{(s)}$ is independent of~$\mathcal{H}_s$, and $X_t^{(s)}\stackrel{d}{=}X_t$. This, together with \eqref{e:phits} and the~\hbox{$\mathcal{H}_s$-measurabilty} of $\Phi_s$, gives
\begin{align*}
\pr_x\left( \Phi_{t+s}\in A \,|\,\mathcal{H}_s \right)(\omega)
&=\ew_x\left[\mathbbm{1}_A(\Phi_{t+s})\,|\,\mathcal{H}_s \right](\omega)
=\ew_x\left[\mathbbm{1}_A\left(\Phi_s \exp\left(-\alpha t+ X_{t}^{(s)}\right)\right)\,|\,\mathcal{H}_s \right](\omega)\\
&=\ew_x\left[\mathbbm{1}_A\left(\Phi_s(\omega) \exp\left(-\alpha t+ X_{t}^{(s)}\right)\right) \right]\\
&=\ew_{\Phi_s(\omega)}\left[ \mathbbm{1}_A\left(\Phi_0 \exp\left(-\alpha t+ X_{t}\right)\right)\right]=\ew_{\Phi_s(\omega)}\left[\mathbbm{1}_A(\Phi_t) \right]=\pr_{\Phi_s(\omega)}\left(\Phi_t\in A\right)
\end{align*}
for $\pr_x$-a.e. $\omega\in\Omega$. Finally, since $\mathcal{F}_s\subset\mathcal{H}_s$ and $\pr_{\Phi_s}\left(\Phi_t\in A\right)$ is $\mathcal{F}_s$-measurable, it follows that
\begin{align*}
\pr_x\left( \Phi_{t+s}\in A \,|\,\mathcal{F}_s \right)
=\ew_x\left[\pr_x\left( \Phi_{t+s}\in A \,|\,\mathcal{H}_s \right) \,|\,\mathcal{F}_s \right]=\ew_x\left[ \pr_{\Phi_s}\left(\Phi_t\in A\right) \, | \, \mathcal{F}_s\right]=\pr_{\Phi_s}\left(\Phi_t\in A\right)\;\text{a.s}.
\end{align*}
Additionally, $\Phi$ is progressively measurable since its sample paths are right-continuous. This follows from the fact that $\{X_t\}_{t\geq 0}$ is C\`{a}dl\`{a}g, which implies that $\Phi$ is C\`{a}dl\`{a}g too.

In light of the above, $\Phi$ can be viewed as a Markov process with transition semigroup $\{P_t\}_{t\geq 0}$ given~by
$$P_t(x,A)=\pr_x(\Phi_t\in A)=\pr\left( x\exp\left(-\alpha t+X_t\right)\in A \right)\quad \text{for}\quad t\geq 0,\;x\in \mathbb{R},\; A\in\mathcal{B}(\mathbb{R}).$$
Clearly, for any $t\geq 0$, $x\in \mathbb{R}$, and bounded below
Borel measureable function \hbox{$f:\mathbb{R}\to\mathbb{R}$},
\begin{equation}\label{e:P_tf-ex}
P_t f(x)=\ew\left[f\left(x\exp\left(-\alpha t+X_t\right)\right)\right].
\end{equation}
In what follows, we will prove that $\{P_t\}_{t\geq 0}$ fulfills hypotheses \ref{cnd:a1}-\ref{cnd:a3}. 

Conditions \ref{cnd:a1} and \ref{cnd:a2} follow immediately from \eqref{e:P_tf-ex} by using the a.s. convergence $X_t\to 0$ as~$t\to 0^+$ (in the case of \ref{cnd:a1}) and the Lebesgue dominated convergence theorem.

To verify \ref{cnd:a3}, we first compute $\ew \exp(X_t)$. For any $t>0$, as $N_t$ has the Poisson distribution with parameter $t$, we get 
\begin{align}\label{e:EexpXt}
\begin{split}
\ew\exp(X_t)&=\pr(N_t=0)+\sum_{n=1}^{\infty}\ew\left[\mathbbm{1}_{\{ N_t=n\}}\prod_{k=1}^n \exp(Y_k) \right]\\&
=\pr(N_t=0)+\sum_{n=1}^{\infty}\pr(N_t=n)\left(\ew \exp(Y_1)\right)^n 
=\exp(-t)\sum_{n=0}^{\infty} \frac{\left(t\,\ew \exp(Y_1)\right)^n}{n!}\\
&=\exp\big((\ew \exp(Y_1)-1)t\big)=\exp\left(\left(\frac{1}{2}e^K+\frac{1}{2}-1 \right)t\right)=\exp\left(\frac{1}{2}\left(e^K-1\right)t \right)\\
&=\exp((\alpha-\gamma)t),
\end{split}
\end{align}
where the last equality follows from the definition of $\alpha$. Of course, the formula is also valid with $t=0$, since $X_0=0$. Given this, it suffices to observe that
\begin{align*}
\<V,\mu P_t\>&=\<P_t V,\mu\>=\int_{\mathbb{R}} \ew |x\exp(-\alpha t+X_t)|\,\mu(dx)=\exp(-\alpha t)\ew \exp(X_t)\int_{\mathbb{R}}|x|\,\mu(dx)\\
&=\exp(-\gamma t)\<V,\mu\><\infty\quad\text{for any}\quad \mu\in\mathcal{M}_{1,1}(\mathbb{R}),\; t\geq 0,
\end{align*}
and that
\begin{align*}
\left|\<f,\, \delta_x P_t- \delta_y P_t \> \right|&=\left|P_t f(x)-P_t f(y) \right|\leq \ew \left|f\left(x\exp\left(-\alpha t+X_t\right)\right)-f\left(y\exp\left(-\alpha t+X_t\right)\right) \right|\\
&\leq \exp(-\alpha t)\left(\ew \exp X_t\right)|x-y|=\exp(-\gamma t)|x-y|
\end{align*}
for any $f\in L(\mathbb{R})$ with $\lip f\leq 1$, $x,y\in \mathbb{R}$, and $t\geq 0$. 

Finally, it remains to examine hypothesis \ref{cnd:a4}. Arguing analogously as in~\eqref{e:EexpXt}, one concludes that
\begin{align*}
\ew\exp(\zeta X_t)&=\exp\big((\ew \exp(\zeta Y_1)-1)t\big)=\exp\left(\frac{1}{2}\left(e^{\zeta K}-1\right)t \right)\quad\text{for}\quad t\geq 0,
\end{align*}
which gives
\begin{align}\label{e:P_tVzeta}
\begin{split}
P_t V^{\zeta}(x)&=\ew\left[\left|x\exp(-\alpha t+X_t)  \right|^{\zeta}\right]=|x|^{\zeta}\exp(-\alpha \zeta t)\ew\exp(\zeta X_t)\\
&=|x|^{\zeta}\exp\left(\left(\frac{1}{2}\left(e^{\zeta K}-1 \right) -\alpha\zeta\right)t \right)=|x|^{\zeta}\exp(\beta t)\quad\text{for}\quad x\in \mathbb{R},\;t\geq 0,
\end{split}
\end{align}
where $\beta:=\left(e^{\zeta K}-1 \right)/2-\alpha\zeta$. On the other hand,  \eqref{e:large_M} ensures that $\beta>0$, since 
\begin{align*}
\beta&=\frac{1}{2}\left(e^{\zeta K}-1 \right)-\alpha\zeta=\frac{1}{2}\left(e^{\zeta K}-1 \right)-\left(\frac{1}{2}\left(e^K-1\right)+\gamma \right)\zeta\\
&=\frac{1}{2}\left(e^{\zeta K}-\zeta e^K \right)-\zeta\gamma+\frac{1}{2}(\zeta-1).
\end{align*}
This shows that $\mu=\delta_0$ is the only Borel probability measure for which condition~\ref{cnd:a4} is satisfied.
\end{example}

Before stating our main result, we will show that, under assumptions \ref{cnd:a2}-\ref{cnd:a4}, the process $\{\Phi_t\}_{t\geq 0}$ possesses a unique exponentially attracting invariant distribution, which has a finite $\zeta $-th moment. While the proof of this fact closely resembles that of \hbox{\cite[Lemma 3.1]{b:czapla-CLT-cont}}, we provide it here to account for the slightly different assumptions in the current framework.

\begin{proposition}\label{prop:1}
Suppose that conditions \ref{cnd:a2}-\ref{cnd:a4} hold. Then $\{P_t\}_{t\geq 0}$ possesses a~unique invariant probability measure $\mu_*$, which belongs to $\mathcal{M}_{1,\zeta}(X)$. Moreover, $\{P_t\}_{t\geq 0}$ is then asymptotically stable, i.e., $\nu P_t\stackrel{w}{\to}\mu_*$ for every $\nu\in\mathcal{M}_1(X)$, and there exists $C>0$ such that
\begin{equation}\label{e:ergodicity}
d_W (\nu P_t, \mu_*)\leq C e^{-\gamma t} (\<V,\nu\>+1)\quad\text{for any}\quad t\geq 0,\; \nu\in\mathcal{M}_{1,1}(X).
\end{equation}
\end{proposition}
\begin{proof}
Keeping in mind Remark \ref{rem:1}, we see that
$$
d_W(\mu P_t,\mu P_{t+s})\leq e^{-\gamma t} (\<V,\mu\>+ \<V,\mu P_s\>)\leq 2e^{-\gamma t}c_1(\mu)\quad\text{for any}\quad s,t\geq 0,
$$
and $c_1(\mu)<\infty$ by Remark \ref{rem:2}. This, in particular, gives
\begin{equation}
\label{e1:1}
d_{BL}(\mu P_{t+n},\,\mu P_{t+n+m})\leq 2e^{-\gamma(t+n)}c_1(\mu)\quad\text{for all}\quad t\geq 0,\;n,m\in\n,
\end{equation}
\begin{equation}
\label{e1:2}
d_{BL}(\mu P_{s+n},\,\mu P_{t+n})\leq 2e^{-\gamma((s\wedge t)+n)}c_1(\mu)\quad\text{for all}\quad s,t\geq 0,\;n\in\n.
\end{equation}
From \eqref{e1:1} it follows that $\{\mu P_{t+n}\}_{n\in\n}$, $t\geq 0$, are Cauchy sequences in the space $(\mathcal{M}_1(X), d_{BL})$ (which is complete), and thus they are convergent, while \eqref{e1:2} shows that all of them have the same limit. Consequently, there exists $\mu_*\in\mathcal{M}_1(X)$ such that $\mu P_{t+n}\stackrel{w}{\to}\mu_*$, as $n\to \infty$, for every $t\geq 0$. 

Condition \ref{cnd:a2}, in turn, guarantees that the measure $\mu_*$ is invariant for $\{P_t\}_{t\geq 0}$, since for any \hbox{$f\in BC(X)$} and $t\geq 0$,
$$\<f,\, \mu_* P_t\>=\<P_t f,\, \mu_*\>=\lim_{n\to \infty}\<P_t f,\, \mu P_n\>=\lim_{n\to \infty}\<f,\, \mu P_{t+n}\>=\<f,\mu_*\>.$$

Furthermore, note that $\mu_*\in\mathcal{M}_{1,\zeta}(X)$. Indeed, hypothesis \ref{cnd:a4} yields that
$$\<V^{\zeta}\wedge k, \, \mu_*\>=\lim_{n\to\infty}\<V^{\zeta}\wedge k, \,\mu P_n\>\leq c_{\zeta}(\mu)<\infty\quad\text{for all}\quad k\in\n,$$
where the equality follows from the fact that $V^{\zeta}\wedge k\in BC(X)$, $k\in\n$. Since $V^{\zeta}\wedge k \uparrow V^{\zeta}$ point-wisely as $k\to\infty$, by using the Lebesgue monotone convergence theorem, we can conclude that 
$$\<V^{\zeta},\mu_*\>=\lim_{k\to \infty} \<V^{\zeta}\wedge k, \, \mu_*\>\leq c_{\zeta}(\mu)<\infty.$$

Obviously, in view of the latter observation and Remark \ref{rem:1}, \hbox{$\lim_{t\to\infty} d_{BL}(\delta_x P_t ,\mu_*)=0$} for every $x\in X$. On the other hand, it is easy to see that
$$
d_{BL}(\nu P_t, \mu_*)\leq \int_X d_{BL}(\delta_x P_t,\mu_*)\,\nu(dx)\quad \text{for any}\quad \nu\in\mathcal{M}_1(X).$$
Therefore, applying the Lebesgue’s dominated convergence theorem, we can deduce that $\nu P_t\stackrel{w}{\to} \mu_*$ for every $\nu\in\mathcal{M}_1(X)$, which simultaneously shows that $\{P_t\}_{t\geq 0}$ has exactly one invariant probability measure.

Finally, from Remark \ref{rem:1} it follows that for any $t\geq 0$ and $\nu\in\mathcal{M}_{1,1}(X)$,
\begin{align*}
d_W (\nu P_t, \mu_*)=d_W (\nu P_t, \mu_* P_t)\leq e^{-\gamma t}(\<V, \nu\>+\<V,\mu_*\>)\leq (\<V,\mu_*\>+1)e^{-\gamma t}(\<V,\nu\>+1),
\end{align*}
which shows that \eqref{e:ergodicity} holds with $C:=\<V,\mu_*\>+1$. The proof is now complete.
\end{proof}
In the remainder of the paper, upon assuming conditions \ref{cnd:a2}-\ref{cnd:a4}, the unique invariant probability measure of $\{P_t\}_{t\geq 0}$ (which belongs to $\mathcal{M}_{1,\zeta}(X)$) will be denoted by $\mu_*$, and $g\in BL(X)$ will stand for an arbitrary bounded Lipschitz continuous function such that $\<g,\mu_*\>=0$. Moreover, for convenience, we put
\begin{equation}
\label{def:I}
I_t(g):=\int_0^t g(\Phi_s)\,ds\quad\text{for}\quad t\geq 0.
\end{equation}
It is worth emphasizing here that, by the assumed progressive measurability of $\Phi$ and Fubini’s theorem,  $\{I_t(g)\}_{t\geq 0}$ is adapted to the filtration $\{\mathcal{F}_t\}_{t\geq 0}$. 

Further, note that, as a consequence of Proposition \ref{prop:1}, we get
\begin{equation}
\label{e:P_tgx}
|P_t g(x)|=|\<g, \delta_x P_t\>-\<g,\mu_*\>|\leq (\lip g)\, d_W(\delta_x P_t, \mu_*)\leq C(\lip g) e^{-\gamma t}(V(x)+1)
\end{equation}
for any $t\geq 0$ and $x\in X$. This allows us to define the mapping $\chi_g:X\to\mathbb{R}$, often referred to as a \emph{corrector function}, by
$$\chi_g(x):=\int_0^{\infty} P_t g(x)\,dt\quad\text{for}\quad x\in X.$$

\begin{remark}\label{rem:lip_chi}
By condition \ref{cnd:a3} the function $\chi_g$ is Lipschitz continuous, since
\begin{align*}
|\chi_g(x)-\chi_g(y)|&\leq \int_0^{\infty} |\<g,\,\delta_x P_t\>-\<g,\,\delta_y P_t\>|\,dt 
\leq \lip g \left(\int_0^{\infty} e^{-\gamma t}\, dt \right)\rho(x,y)\\
&=\frac{\lip g}{\gamma}\rho(x,y) \quad\text{for any} \quad x,y\in X.
\end{align*}
\end{remark}
Finally, let us introduce
\begin{equation}
\label{def:sigma}
\sigma_g^2:=\ew_{\mu_*}\left[\left(\chi_g(\Phi_1)-\chi_g(\Phi_0)+I_1(g)\right)^2\right].
\end{equation}
At the end of the paper (in Section \ref{sec:5}), we will demonstrate (assuming also \ref{cnd:a1}) that $\sigma_g^2$ can be equivalently expressed in the following forms:
\begin{equation}
\label{e:sigma_eq}
\sigma_g^2=2\<g\chi_g ,\mu_*\>=\lim_{t\to \infty}\frac{\ew_{\mu}\left[I_t^2(g) \right]}{t}.
\end{equation}
The second equality clearly indicates that $\sigma_g^2$ is the asymptotic variance of the process $\{I_t(g)\}_{t \geq 0}$. It is worth noting that identifying reasonable conditions to guarantee the positivity of the asymptotic variance in a general framework is inherently challenging. Such results are typically achievable only when analyzing specific models. Nevertheless, it can be shown, for example, that $\sigma_g^2 > 0$ if the measures $\delta_x P_t$ and $\mu_*$ are mutually absolutely continuous for all \hbox{$(t, x) \in [0, \infty) \times X$}, and $g \neq 0$ on some Borel subset of $X$ with positive $\mu_*$-measure (see~\hbox{\cite[Proposition 2.4]{b:bhattacharya})}.

Our main result can be now stated as follows.
\begin{theorem}\label{thm:main}
Suppose that conditions \ref{cnd:a1}-\ref{cnd:a4} hold. Then $\sigma_g^2$, given by \eqref{def:sigma}, is \hbox{finite} and, whenever it is positive, the process $\{I_t(g)\}_{t\geq 0}$, defined by \eqref{def:I}, satisfies the LIL under~$\pr_{\mu}$, with~$\sigma_g^2$ serving as its asymptotic variance. Specifically,
$$\liminf_{t\to\infty}\frac{I_t(g)}{\sqrt{2\sigma_g^2\, t\ln\ln t}}=-1
\quad \text{and}\quad \limsup_{t\to\infty}\frac{I_t(g)}{\sqrt{2\sigma_g^2 \,t\ln\ln t}}=1\;\;\;\pr_{\mu}-a.s.$$
\end{theorem}
Notably, Example \ref{example:1} provides a nontrivial Markov process that satisfies all the assumptions of the theorem above with each $\mu\in\mathcal{M}_{1,\zeta}(\mathbb{R})$, whenever $\zeta>2$ and $K,\gamma>0$ are chosen so that $\beta=(e^{\zeta K}-\zeta e^K)/2-\zeta\gamma+(\zeta-1)/2\leq 0$. Indeed, in view of~\eqref{e:P_tVzeta}, this choice ensures that condition~\ref{cnd:a4} is satisfied, while the remaining assumptions are already verified within the example. Although this model can be naturally generalized in several directions, as already mentioned in the introduction, a more substantial application of our result will be presented in Section~\ref{sec:6}.

\section{Proof of the Main Result} \label{sec:3}
In this section, we present a number of auxiliary results that ultimately enable us to prove Theorem \ref{thm:main}. We assume throughout that all hypotheses \ref{cnd:a1}–\ref{cnd:a4} are fulfilled and adhere to the notation introduced in the previous parts of the paper. In particular, $\mu$ stands for the initial distribution of $\Phi$ satisfying \ref{cnd:a4}, $\mu_*$ denotes the unique invariant probability measure of the semigroup~$\{P_t\}_{t\geq 0}$, and $g$ is a bounded, Lipschitz continuous function such that $\<g,\mu_*\>=0$.

It should be emphasized, however, that hypothesis \ref{cnd:a1} is not required for the lemmas presented in Subsection \ref{sec:3.1}. This condition becomes necessary in Subsection~\ref{sec:3.2}, where it is specifically needed for Lemma~\ref{lem:5} to hold and, consequently, it must also be assumed for all subsequent results (i.e., Lemmas~\ref{lem:6}–\ref{lem:9}) that rely on it. The proof of Lemma~\ref{lem:5} itself, being both lengthy and technical, is postponed to Section~\ref{sec:4}.

\subsection{Discretization and a Martingale-Based Decomposition}\label{sec:3.1}
We begin with a result that reduces the proof of Theorem \ref{thm:main} to verifying the LIL (under~$\pr_{\mu}$) for the sequence $\{I_n(g)\}_{n\in\n_0}$. A comparable approach is employed, for instance, in the proof of \hbox{\cite[Theorem 2.1]{b:komorowski_LIL}}, which focuses on some specific model.
\begin{lemma}\label{lem:1}
We have
$$\lim_{\hspace{-0.2cm}\mathbb{N}\ni\, n\to\infty}\, \sup_{t\in [n,n+1)} \left|\frac{I_t(g)}{\sqrt{t\ln\ln t}}-\frac{I_n(g)}{\sqrt{n\ln\ln n}} \right|=0\;\;\;\pr_{\mu}-\text{a.s.}$$
\end{lemma}

\begin{proof}
First of all, observe that, for every $n\geq e^e$,
\begin{align*}
\frac{1}{\sqrt{n\ln\ln n}}-\frac{1}{\sqrt{(n+1)\ln\ln(n+1)}}
&\leq \frac{\sqrt{(n+1)\ln\ln(n+1)}-\sqrt{n\ln\ln n}}{n}\\
&=\frac{(n+1)\ln\ln(n+1)-n\ln\ln n }{n\left( \sqrt{(n+1)\ln\ln(n+1)}+\sqrt{n\ln\ln n}\right) }\\
&\leq \frac{(n+1)\ln\ln(n+1)-n\ln\ln n }{n^{3/2}}.
\end{align*}

Now, fix an arbitrary $\varepsilon>0$ and choose $n_0\in\n$ large enough, so that 
$$n_0\geq e^e\quad\text{and}\quad \frac{\norma{g}_{\infty}}{\sqrt{n_0\ln\ln(n_0)}}<\frac{\varepsilon}{2}.$$ 
Further, let $n\geq n_0$. Taking into account the first observation, we see that, for every \hbox{$t\in [n,n+1)$},
\begin{align*}
\left|\frac{I_t(g)}{\sqrt{t\ln\ln t}}-\frac{I_n(g)}{\sqrt{n\ln\ln n}}\right|
&\leq \frac{|I_t(g)-I_n(g)|}{\sqrt{t\ln\ln t}}+|I_n(g)|\left(\frac{1}{\sqrt{n\ln\ln n}}-\frac{1}{\sqrt{t\ln\ln t}} \right)\\
&\leq \frac{\norma{g}_{\infty}}{\sqrt{n\ln\ln n}}+|I_n(g)|\left(\frac{1}{\sqrt{n\ln\ln n}}-\frac{1}{\sqrt{(n+1)\ln\ln(n+1)}} \right)\\
&< \frac{\varepsilon}{2}+|I_n(g)|\frac{(n+1)\ln\ln(n+1)-n\ln\ln n }{n^{3/2}}.\\
\end{align*}
Consequently, applying the Markov inequality, we get
\begin{align}
\begin{split}
\label{e2:0}
\pr_{\mu}\left(\sup_{t\in[n,n+1)}\left|\frac{I_t(g)}{\sqrt{t\ln\ln t}}-\frac{I_n(g)}{\sqrt{n\ln\ln n}}\right|\geq \varepsilon \right)\hspace{-3cm}\\
&\leq\pr_{\mu}\left( |I_n(g)|\frac{(n+1)\ln\ln(n+1)-n\ln\ln n }{n^{3/2}}\geq\frac{\varepsilon}{2}\right)\\
&=\pr_{\mu}\left(|I_n(g)|\geq \frac{\varepsilon}{2}\cdot \frac{n^{3/2}}{(n+1)\ln\ln(n+1)-n\ln\ln n }\right)\\
&\leq \frac{4 \ew_{\mu} \left[I_n^2(g)\right] }{\varepsilon^2}\cdot \frac{\left[(n+1)\ln\ln(n+1)-n\ln\ln n\right]^2 }{n^3}.
\end{split}
\end{align}

We shall now estimate the expectation $\ew_{\mu} \left[I_n^2(g)\right]$. From the Markov property, it follows that, for every $x\in X$,
\begin{align}
\label{e:int_square}
\begin{split}
\ew_x[I_n^2(g)]&=2\int_0^n\int_0^s \ew_x[g(\Phi(u)) g(\Phi(s))]\,du\,ds\\
&=2\int_0^n\int_0^s \ew_x[g(\Phi(u))\ew_x\left[g(\Phi(s))\,|\,\mathcal{F}_u]\right] du\,ds\\
&=2\int_0^n \int_0^s \ew_x\left[g(\Phi(u))\ew_{\Phi(u)}[g(\Phi(s-u))]\right] du\,ds\\
&=2\int_0^n \int_0^s \ew_x\left[(g\cdot P_{s-u}g)(\Phi(u))\right]du\,ds
= 2\int_0^n \int_0^s P_u(g\cdot P_{s-u}g)(x)\,du\,ds.
\end{split}
\end{align}
Moreover, in view of \eqref{e:P_tgx}, for any $0\leq u\leq s$ we have
\begin{align*}
|P_u(g\cdot P_{s-u} g)|&\leq P_u(|g|\cdot |P_{s-u}g|)\leq \norma{g}_{\infty} P_u(|P_{s-u}g|)\\
&\leq C\norma{g}_{\infty}(\lip g)e^{-\gamma(s-u)}(P_u V+1)\leq C\norma{g}_{BL}^2e^{-\gamma(s-u)}(P_u V+1).
\end{align*}
Thus, using the Fubini theorem (together with the measurability of $(u,x)\mapsto P_u g(x)$), we can conclude that
\begin{align*}
\ew_{\mu}[I_n^2(g)]&\leq 2\int_X\left(\int_0^n \int_0^s |P_u(g\cdot P_{s-u}g)(x)|\,du\,ds
\right)\mu(dx)
\\&
\leq 2C\norma{g}_{BL}^2\int_X\left(\int_0^n \int_0^s e^{-\gamma(s-u)}(P_u V(x)+1)\,du\,ds \right)\mu(dx)\\
&=2C\norma{g}_{BL}^2\int_0^n \int_0^s e^{-\gamma(s-u)}(\<V,\mu P_u\>+1)\,du\,ds\\
&\leq 2C\norma{g}_{BL}^2(c_1(\mu)+1)\int_0^n \int_0^s e^{-\gamma(s-u)}\,du\,ds
\\&=  2C\norma{g}_{BL}^2(c_1(\mu)+1)\frac{1}{\gamma}\left(n-\frac{1-e^{-\gamma n}}{\gamma}\right)\leq D(\mu)n,
\end{align*}
where $$D(\mu):=\frac{2C\norma{g}_{BL}^2}{\gamma}(c_1(\mu)+1)<\infty,$$
due to Remark \ref{rem:2}.

Based on the estimate derived above, \eqref{e2:0} provides that
\begin{align*}
\pr_{\mu}\left(\sup_{t\in[n,n+1)}\left|\frac{I_t(g)}{\sqrt{t\ln\ln t}}-\frac{I_n(g)}{\sqrt{n\ln\ln n}}\right|\geq \varepsilon \right)&\leq \frac{4 D(\mu) }{\varepsilon^2}\cdot \frac{\left[(n+1)\ln\ln(n+1)-n\ln\ln n\right]^2 }{n^2}.
\end{align*}

In view of the Borel-Cantelli Lemma, it now suffices to show that
$$\sum_{n=n_0}^{\infty}\left(\frac{(n+1)\ln\ln(n+1)-n\ln\ln n}{n}\right)^2<\infty.$$
To this end, observe that
\begin{align*}
[(n+1)\ln\ln(n+1)-n\ln\ln n]^2&= [n(\ln\ln(n+1)-\ln\ln n)+\ln\ln(n+1)]^2\\
&\leq 2[n(\ln\ln(n+1)-\ln\ln n)]^2+2\ln^2\ln(n+1)\\
&\leq 2 \left[n\ln\left(\frac{\ln(n+1)}{\ln n}\right)\right]^2+2\ln^2 n\quad\text{for}\quad n\geq n_0.
\end{align*}
Additionally, we have
$$\lim_{n\to\infty} n\ln\left(\frac{\ln(n+1)}{\ln n}\right)=0,$$
since, for every $n\geq n_0$,
$$1\leq \frac{\ln(n+1)}{\ln n}=\frac{\ln(n+1)-\ln n}{\ln n}+1=\frac{\ln\left(1+1/n\right)}{\ln n}+1\leq \frac{1}{n\ln n}+1,$$
which gives
$$0\leq n\ln\left(\frac{\ln(n+1)}{\ln n}\right)\leq n \ln\left(\frac{1}{n\ln n}+1 \right)\leq n\cdot\frac{1}{n\ln n}=\frac{1}{\ln n}.$$
Hence, in particular, there exists $M\geq 0$ such that
$$0\leq n\ln\left(\frac{\ln(n+1)}{\ln n}\right) \leq M \quad \text{for all}\quad n\geq n_0,$$
and, therefore,
$$[(n+1)\ln\ln(n+1)-n\ln\ln n]^2\leq 2M^2+2\ln^2(n)\leq 2(M^2+1)\ln^2 n\quad\text{for all}\quad n\geq n_0.$$
Finally, it follows that
$$0\leq \sum_{n=n_0}^{\infty}\left(\frac{(n+1)\ln\ln(n+1)-n\ln\ln n }{n}\right)^2\leq 2(M^2+1)\sum_{n=n_0}^{\infty} \frac{\ln^2 n}{n^2}<\infty,$$
where the series on the right-hand side converges by Cauchy's condensation test and the root test, as
$$\sum_{n=n_0}^{\infty} \frac{2^n \ln^2(2^n)}{(2^n)^2}=\ln^2(2)\sum_{n=n_0}^{\infty} \frac{n^2}{2^n}.$$
\end{proof}

\begin{remark}
Note that the key idea underlying the proof of Lemma \ref{lem:1} is that, owing to~\eqref{e:P_tgx} and our estimation approach, we have been able to improve the asymptotic order of~$\ew_{\mu}[I_n^2(g)]$ as $n \to \infty$ from $\mathcal{O}(n^2)$ under a `naive' estimation to $\mathcal{O}(n)$.
\end{remark}

Let us now introduce the processes
\begin{equation}
\label{def:martingale}
M_t(g):=\chi_g(\Phi_t)-\chi_g(\Phi_0)+I_t(g)\quad\text{for}\quad t\geq 0,
\end{equation}
\begin{equation}
\label{def:rest}
R_t(g):=\frac{\chi_g(\Phi_0)-\chi_g(\Phi_t)}{\sqrt{2 t \ln\ln t}}\quad\text{for}\quad t>e.
\end{equation}
Then
\begin{equation}
\label{e:decomposition}
\frac{I_t(g)}{\sqrt{2t \ln\ln t}}=\frac{M_t(g)}{\sqrt{2t \ln\ln t}}+R_t(g)\quad \text{for}\quad t>e,
\end{equation}
and \eqref{def:sigma} can be rewritten as
\begin{equation}
\label{e:sigma_short}
\sigma_g^2=\ew_{\mu_*}[M_1(g)^2].
\end{equation}

\begin{remark}\label{rem:4} There exists a constant $A_{\zeta}\geq 0$ such that
$$\sup_{t\geq 0} \ew_{\nu}\left[|\chi_g(\Phi_t)|^{\zeta}\right]  \leq A_{\zeta}(c_{\zeta}(\nu)+1)<\infty\quad \text{for}\quad \nu\in\{\mu,\,\mu_* \}.$$
Indeed, using \eqref{e:P_tgx} and the fact that 
\begin{equation}
\label{e:useful_ineq}
|\alpha+\beta|^r\leq (2^{r-1}\vee 1)(|\alpha|^r+|\beta|^r)\quad\text{for any}\quad \alpha,\beta\in\mathbb{R},\;r>0,
\end{equation}
we get
\begin{align*}
\ew_{\nu}\left[|\chi_g(\Phi_t)|^{\zeta}\right]&\leq \int_X\left(\int_0^{\infty} |P_s g(x)|\,ds \right)^{\zeta}\nu P_t(dx)\\
&\leq \int_X \left(C(\lip g) (V(x)+1)\int_0^{\infty} e^{-\gamma s}\,ds \right)^{\zeta}\nu P_t(dx)\\
&=\left(\frac{C\lip g}{\gamma} \right)^{\zeta}\<(V+1)^{\zeta},\nu P_t\>\leq A_{\zeta}\left(\<V^{\zeta},\nu P_t\>+1\right)\leq A_{\zeta}(c_{\zeta}(\nu)+1),
\end{align*}
where $A_{\zeta}:=2^{\zeta-1}(C(\lip g)/\gamma)^{\zeta}$. Of course, $c_{\zeta}(\mu)$ and $c_{\zeta}(\mu_*)$ are finite by condition \ref{cnd:a4} and the fact that $\mu_*$ is a $\{P_t\}_{t\geq 0}$-invariant measure, satisfying $\<V^{\zeta},\mu_*\><\infty$.
\end{remark}

\begin{lemma}\label{lem:3}
For $\nu\in\{\mu,\mu_*\}$ and each $t\geq 0$, we have $\ew_{\nu}\left[|M_t(g)|^{\zeta}\right]<\infty$ and, in particular, the constant $\sigma_g^2$, defined by \eqref{def:sigma}, is finite.
\end{lemma}

\begin{proof}
Let $t\geq 0$. From \eqref{e:useful_ineq}, it follows that
$$|M_t(g)|^{\zeta}\leq 4^{\zeta-1}\left(|\chi_g(\Phi_t)|^{\zeta}+ |\chi_g(\Phi_0)|^{\zeta}\right)+2^{\zeta-1}(\norma{g}_{\infty}t)^{\zeta}.$$
Thus, in view of Remark \ref{rem:4}, there exists $A_{\zeta}\in [0,\infty)$ such that
$$\ew_{\nu}\left[|M_t(g)|^{\zeta}\right]\leq 2\cdot 4^{\zeta-1} A_{\zeta}\left(c_{\zeta}(\nu)+1\right)+2^{\zeta-1}(\norma{g}_{\infty}t)^{\zeta}<\infty,$$
which completes the proof.
\end{proof}

\begin{lemma}\label{lem:2}
The sequence $\{R_n(g)\}_{n\geq 3}$, determined by \eqref{def:rest}, converges to zero $\pr_{\mu}$-a.s.
\end{lemma}
\begin{proof}
Fix an arbitrary $\varepsilon>0$ and let $n\geq 3$. The Markov inequality implies that
$$
\pr_{\mu}(|R_n(g)|\geq \varepsilon)=\pr_{\mu}(|\chi_g(\Phi_0)-\chi_g(\Phi_n)|\geq\varepsilon\sqrt{2n\ln\ln n})\leq \frac{\ew_{\mu}[|\chi_g(\Phi_0)-\chi_g(\Phi_n)|^{\zeta} ]}{\varepsilon^{\zeta}(2n\ln\ln n)^{\zeta/2}}.
$$
On the other hand, applying \eqref{e:useful_ineq} and Remark \ref{rem:4}, we obtain
$$\ew_{\mu}\left[|\chi_g(\Phi_0)-\chi_g(\Phi_n)|^{\zeta}\right]\leq 2^{\zeta-1}
\left(\ew_{\mu}\left[|\chi_g(\Phi_0)|^{\zeta}\right]+\ew_{\mu}\left[|\chi_g(\Phi_n)|^{\zeta}\right] \right)\leq 2^{\zeta}A_{\zeta}(c_{\zeta}(\mu)+1)$$
for some $A_{\zeta}\in [0,\infty)$. Consequently,
$$
\sum_{n=3}^{\infty}\pr_{\mu}(|R_n(g)|\geq \varepsilon)\leq \frac{2^{\zeta/2}A_{\zeta}(c_{\zeta}(\mu)+1)}{\varepsilon^{\zeta}}\sum_{n=3}^{\infty}\frac{1}{(n\ln\ln n)^{\zeta/2}},
$$
and the series on the right-hand side is convergent, since $\zeta/2>1$. The claim now easily follows from the Borel-Cantelli Lemma.
\end{proof}

In view of \eqref{e:decomposition}, Lemmas \ref{lem:1} and \ref{lem:2} demonstrate that, in order to prove Theorem \ref{thm:main}, it is enough to establish the LIL (under $\pr_{\mu}$) for the sequence $\{M_n(g)\}_{n\in\n_0}$, which means that
\begin{equation}
\label{e:LIL_mart}
\liminf_{n\to\infty} \frac{M_n(g)}{\sqrt{2\sigma_g^2\, n\ln\ln n}}=-1\quad\text{and}\quad \limsup_{n\to\infty} \frac{M_n(g)}{\sqrt{2\sigma_g^2\, n\ln\ln n}}=1\;\;\pr_{\mu}-\text{a.s.},
\end{equation}
whenever $\sigma_g^2>0$ (finiteness of $\sigma_g^2$ is ensured by Lemma \ref{lem:3}).

Furthermore, proceeding similarly as in \cite[Lemma 4.3]{b:czapla-CLT-cont} (cf. also \cite[Proposition~5.2]{b:komorowski}) one can show that the process $\{M_t(g)\}_{t\geq 0}$, given by \eqref{def:martingale}, is a martingale on $(\Omega,\mathcal{F},\pr_{\mu})$ with respect to the filtration $\{\mathcal{F}_t\}_{t\geq 0}$ and, additionally, from Lemma~\ref{lem:3} we know that $M_t(g)\in\mathcal{L}^{\zeta}(\pr_{\mu})$ for $t\geq 0$. Therefore, in particular, the sequence $\{M_n(g)\}_{n\in\n_0}$ is a~martingale in $\mathcal{L}^{\zeta}(\pr_{\mu})$ with respect to $\{\mathcal{F}_n\}_{n\in\n_0}$.

\subsection {Proof of the LIL for the Associated Martingale}\label{sec:3.2}

Let $\{Z_n(g)\}_{n\in\n}$ denote the sequence of the increments of $\{M_n(g)\}_{n\in\n}$, i.e.,
$$Z_n(g):=M_n(g)-M_{n-1}(g)=\chi_g(\Phi_n)-\chi_g(\Phi_{n-1})+\int_{n-1}^n g(\Phi_s)\,ds,\quad\text{for}\quad n\in\n,$$
and define
$$s_n^2(\mu,g):=\ew_{\mu}\left[M_n^2(g)\right]=\sum_{k=1}^n\ew_{\mu}\left[Z_k^2(g)\right]\quad\text{for}\quad n\in\n.$$
Clearly, by Lemma \ref{lem:3}, $s_n^2(\mu,g)<\infty$ for all $n\in\n$.

To prove that the martingale $\{M_n(g)\}_{n\in\n_0}$ satisfies the LIL (under $\mathbb{P}_{\mu}$), we will rely on \cite[Theorem 1]{b:heyde_scott}, as explained in the remark below.

\begin{remark}\label{rem:5}
Assuming that $\sigma_g^2>0$, to show \eqref{e:LIL_mart}, it suffices to verify the following conditions:
\begin{equation}
\label{cnd:b1}
\lim_{n\to\infty}\frac{1}{n}\sum_{k=1}^n Z_k^2(g)=\sigma_g^2\;\;\;\pr_{\mu}-\text{a.s.};
\end{equation}
\begin{equation}
\label{cnd:b2}
\lim_{n\to\infty} \frac{s_n^2(\mu,g)}{n}=\sigma_g^2;
\end{equation}
\begin{equation}
\label{cnd:b3}
\sum_{n=\bar{n}}^{\infty} s_n^{-4}(\mu,g)\,\ew_{\mu}\left[Z_n(g)^4 \mathbbm{1}_{\left\{|Z_n(g)|<\delta s_n(\mu,g)\right\} } \right]<\infty \quad\text{for some}\quad \delta>0;
\end{equation}
\begin{equation}
\label{cnd:b4}
\sum_{n=\bar{n}}^{\infty} s_n^{-1}(\mu,g)\,\ew_{\mu}\left[|Z_n(g)|\mathbbm{1}_{\left\{|Z_n(g)|\geq \varepsilon s_n(\mu,g)\right\} } \right]<\infty \quad\text{for all}\quad \varepsilon>0,
\end{equation}
where $\bar{n}$ is large enough so that $s_n^2(\mu,g)>0$ for all $n\geq \bar{n}$. 

Indeed, \eqref{cnd:b2} clearly yields that 
\begin{equation}
\label{cnd:b5}
\lim_{n\to\infty} s_n^2(\mu,g)=\infty,
\end{equation}
and, combined with \eqref{cnd:b1}, it gives
\begin{equation}
\label{cnd:b6}
\lim_{n\to\infty} \frac{1}{s_n^2(\mu,g)}\sum_{k=1}^n Z_k^2(g)=\lim_{n\to\infty} \frac{n}{s_n^2(\mu,g)}\cdot\frac{1}{n} \sum_{k=1}^n Z_k^2(g)=1\;\;\;\pr_{\mu}-\text{a.s}.
\end{equation}
Conditions \eqref{cnd:b3}-\eqref{cnd:b6} coincide with the hypotheses of \cite[Theorem 1]{b:heyde_scott}, which implies, in~particular, that 
$$
\liminf_{n\to\infty} \frac{M_n(g)}{\sqrt{2 s_n^2(\mu,g) \ln\ln s_n^2(\mu,g)}}=-1,\quad \limsup_{n\to\infty} \frac{M_n(g)}{\sqrt{2 s_n^2(\mu,g) \ln\ln s_n^2(\mu,g)}}=1.\;\;\pr_{\mu}-\text{a.s.}
$$
Consequently, as
$$\frac{M_n(g)}{\sqrt{2\sigma_g^2\, n\ln\ln n}}=\frac{M_n(g)}{\sqrt{2 s_n^2(\mu,g) \ln\ln s_n^2(\mu,g)}}\cdot \sqrt{\frac{s_n^2(\mu,g)}{\sigma_g^2 n}}\cdot\sqrt{\frac{\ln\ln s_n^2(\mu,g)}{\ln\ln n}}\quad\text{for any}\quad n>e,$$
it suffices to observe that
$$\lim_{n\to\infty}\frac{\ln\ln s_n^2(\mu,g)}{\ln\ln n}=1.$$
This, however, follows immediately from \eqref{cnd:b2} and the fact that $0<\sigma_g^2<\infty$. Indeed, letting \hbox{$\alpha_n:=s_n^2(\mu,g)/n$} for $n\in\n$, we obtain
\begin{align*}
\frac{\ln\ln s_n^2(\mu,g)}{\ln\ln n}-1&=\frac{\ln\ln (n\,\alpha_n)-\ln\ln n}{\ln\ln n}=\frac{1}{\ln\ln n}\ln\left(\frac{\ln (n\,\alpha_n)}{\ln n}\right)\\
&=\ln\left(\left(1+\frac{\ln\alpha_n}{\ln n}\right)^{1/\ln\ln n} \right)  \quad\text{for all}\quad n>e,
\end{align*}
which converges to $0$ as $n\to \infty$, since $\lim_{n\to\infty}\alpha_n= \sigma_g^2$.
\end{remark}

In the lemmas that follow, we will show that conditions \eqref{cnd:b1}–\eqref{cnd:b4} are fulfilled.

\begin{lemma}\label{lem:4}
For any $p\in\n\cup\{\infty\}$ and $i\in\{0,1\}$, we have
$$\lim_{n\to\infty} \frac{1}{n}\sum_{k=1}^n\left(Z_{2k+i}^2(g)\wedge p\right)=\ew_{\mu_*}\left[Z_1^2(g)\wedge p\right]\;\;\;\pr_{\mu_*}-\text{a.s.}$$
\end{lemma}
\begin{proof}
Since $\mu_*$ is the unique stationary distribution of $\Phi$, it follows that, for every $t>0$, the shift $\theta_t$ is an $\mathcal{F}/\mathcal{F}$-measurable transformation preserving the measure $\pr_{\mu_*}$ (i.e.,~$\pr_{\mu_*}(\theta_t^{-1}(F))=\pr_{\mu_*}(F)$ for every $F\in\mathcal{F}$), and $\pr_{\mu_*}$ is ergodic with respect to $\theta_t$ (i.e.,~$\pr_{\mu_*}(F)\in\{0,1\}$ for any $F\in\mathcal{F}$ such that $\theta_t^{-1}(F)=F$). Consequently, the Birkhoff individual ergodic theorem implies (see, e.g., \cite[Theorem 7.19]{b:douc}) that for every $t>0$ and any $Y\in\mathcal{L}^2(\pr_{\mu_*})$,
\begin{equation}
\label{e:birkhoff}
\lim_{n\to\infty}\frac{1}{n}\sum_{k=0}^{n-1} Y\circ\theta_t^k=\ew_{\mu_*}[Y]\;\;\;\pr_{\mu_*}-\text{a.s.}
\end{equation}

Now, let $p\in\n\cup\{\infty\}$, $i\in\{0,1\}$ and note that, due to \eqref{m:2}, for every $k\in\n$,
$$Z_{2k+i}^2(g)\wedge p=(Z_{2+i}^2(g)\wedge p)\circ \theta_{2(k-1)}=(Z_{2+i}^2(g)\wedge p)\circ \theta_2^{k-1}.$$
Therefore, applying \eqref{e:birkhoff} with $t=2$ to the random variables $Z_2^2(g)$ and $Z_3^2(g)$ (which belong to $\mathcal{L}^2(\pr_{\mu_*})$ by Lemma \ref{lem:3}) gives
$$\lim_{n\to\infty}\frac{1}{n}\sum_{k=1}^n \left(Z_{2k+i}^2(g) \wedge p\right)=\lim_{n\to\infty}\frac{1}{n}\sum_{k=0}^{n-1} \left(Z_{2+i}^2(g) \wedge p\right)\circ\theta_2^k=\ew_{\mu_*}\left[Z_{2+i}^2(g)\wedge p\right].$$
Finally, referring to the Markov property and the invariance of $\mu_*$ we obtain
\begin{align*}
\ew_{\mu_*}\left[Z_{2+i}^2(g)\wedge p\right]&=\ew_{\mu_*}\left[\ew_{\mu_*}\left[Z_{2+i}^2(g)\wedge p\,|\,\mathcal{F}_{1+i}\right]\right]=\ew_{\mu_*}\left[\ew_{\Phi_{1+i}}\left[Z_1^2(g)\wedge p\right]\right]\\
&=\int_X \ew_x\left[Z_1^2(g)\wedge p \right](\mu_* P_{1+i})(dx)=\ew_{\mu_*}\left[Z_1^2(g)\wedge p \right],
\end{align*}
which completes the proof.
\end{proof}

The next result is crucial to our approach, as it enables transferring the statement of Lemma \ref{lem:4} to the non-stationary case (involving $\mu$ as the initial distribution of $\Phi$). This, in turn, is established in the subsequent Lemma \ref{lem:6}.

\begin{lemma}\label{lem:5}
For any $p\in\n\cup\{\infty\}$, any $d\in [0,\infty)$, and each $i\in\{0,1\}$, the functions  \hbox{$\I{i}{p}{d}, \S{i}{p}{d}:X\to [0,1]$} given by
\begin{gather}
\begin{split}
\label{def:IS}
\I{i}{p}{d}(x):=\ew_x\left[\left|\liminf_{n\to\infty}\left(\left(\frac{1}{n}\sum_{k=1}^n\left(Z_{2k+i}^2(g) \wedge p \right)\right) \right)-d \right|\wedge 1 \right],\\
\S{i}{p}{d}(x):=\ew_x\left[\left|\limsup_{n\to\infty}\left(\left(\frac{1}{n}\sum_{k=1}^n\left(Z_{2k+i}^2(g) \wedge p \right)\right) \right)-d\right|\wedge 1 \right]
\end{split}
\end{gather}
are continuous.
\end{lemma}
As mentioned earlier, the proof of the result above will be provided in Section \ref{sec:4}. 
A~significant aspect of the reasoning outlined there is that the sums appearing in \eqref{def:IS} depend on disjoint segments of the corresponding sample paths of the process $\Phi$ (see Lemma \ref{lem-c:8}). This is the reason for introducing two pairs of functions, involving sums of $Z_{2k}^2(g)\wedge p$ and $Z_{2k+1}^2(g)\wedge p$, rather than simply two functions relying on the sum of $Z_k(g)^2\wedge p$.

\begin{lemma}\label{lem:6}
For each $p\in\n\cup\{\infty\}$, we have
$$\lim_{n\to\infty} \frac{1}{n}\sum_{k=1}^n\left(Z_k^2(g)\wedge p\right)=\ew_{\mu_*}\left[Z_1^2(g)\wedge p\right]\;\;\;\pr_{\mu}-\text{a.s.},$$
and, in particular, condition \eqref{cnd:b1} holds.
\end{lemma}
\begin{proof}
Let $p\in\n\cup\{\infty\}$ and put $d(p):=\ew_{\mu_*}\left[Z_1^2(g)\wedge p\right]$. Moreover, let $\I{i}{p}{d(p)}$ and~$\S{i}{p}{d(p)}$ be the maps defined according to \eqref{def:IS}. Then, by virtue of Lemma \ref{lem:4}, we infer that
\begin{equation}
\label{e3:1}
\<\I{i}{p}{d(p)},\,\mu_*\>=\<\S{i}{p}{d(p)},\,\mu_*\>=0\quad\text{for}\quad i\in\{0,1\}.
\end{equation}

Further, fix $i\in\{0,1\}$ and define
\begin{gather*}
Y_{\mathcal{I}}:=\left|\liminf_{n\to\infty}\left(\left(\frac{1}{n}\sum_{k=1}^n\left(Z_{2k+i}^2(g) \wedge p \right)\right) \right)-d(p) \right|\wedge 1,\\
Y_{\mathcal{S}}:=\left|\limsup_{n\to\infty}\left(\left(\frac{1}{n}\sum_{k=1}^n\left(Z_{2k+i}^2(g) \wedge p \right)\right) \right)-d(p) \right|\wedge 1.
\end{gather*}
From the fact that $Y_{\mathcal{I}}=Y_{\mathcal{I}}\circ \theta_m$ for every $m\in\n$ and \eqref{e:mar_gen} it follows that
\begin{align*}
\ew_{\mu}Y_{\mathcal{I}}&=\ew_{\mu}\left[\ew_{\mu}\left[Y_{\mathcal{I}}\circ\theta_m\,|\,\mathcal{F}_m\right]\right]=\ew_{\mu}\left[\ew_{\Phi_m} Y_{\mathcal{I}} \right]=\int_X \ew_x Y_{\mathcal{I}}\,(\mu P_m)(dx)\\
&=\<\I{i}{p}{d(p)},\,\mu P_m\>\quad\text{for all}\quad m\in\n.
\end{align*}
Analogously, we get
$$
\ew_{\mu}Y_{\mathcal{S}}=\<\S{i}{p}{d(p)},\,\mu P_m\>\quad\text{for all}\quad m\in\n.$$
On the other hand, Lemma \ref{lem:5} guarantees the functions $\I{i}{p}{d(p)}$, $\S{i}{p}{d(p)}$ are continuous, and thus they belong to $BC(X)$. Hence, as the semigroup $\{P_t\}_{t\geq 0}$ is asymptotically stable (by~Proposition \ref{prop:1}), keeping in mind \eqref{e3:1}, we can conclude that
\begin{gather*}
\ew_{\mu} Y_{\mathcal{I}}=\lim_{m\to\infty} \<\I{i}{p}{d(p)},\,\mu P_m\>=\<\I{i}{p}{d(p)},\,\mu_*\>=0,\\
\ew_{\mu} Y_{\mathcal{S}}=\lim_{m\to\infty} \<\S{i}{p}{d(p)},\,\mu P_m\>=\<\S{i}{p}{d(p)},\,\mu_*\>=0.
\end{gather*}
This implies that $Y_{\mathcal{I}}=Y_{\mathcal{S}}=0$ $\pr_{\mu}$ - a.s., which, in turn, yields that
$$
\liminf_{n\to\infty}\frac{1}{n}\sum_{k=1}^n\left(Z_{2k+i}^2(g) \wedge p \right)=\limsup_{n\to\infty}\frac{1}{n}\sum_{k=1}^n\left(Z_{2k+i}^2(g) \wedge p \right)=d(p)\;\;\;\pr_{\mu}-a.s.
$$
We have, therefore, shown that
\begin{equation}
\label{e3:2}
\lim_{n\to\infty}\frac{1}{n}\sum_{k=1}^n\left(Z_{2k+i}^2(g) \wedge p \right)=d(p)\;\;\;\pr_{\mu}\text{ - a.s.}
\end{equation}

Now, to complete the proof, it remains to observe that \eqref{e3:2}, applied with $i\in\{0,1\}$, gives us
\begin{align*}
\lim_{n\to\infty} \frac{1}{2n}\sum_{k=1}^{2n} \left(Z_k^2(g)\wedge p\right)
&=\lim_{n\to\infty} \frac{1}{2n}\left(\sum_{k=1}^n \left(Z_{2k}^2(g)\wedge p\right) +\sum_{k=1}^n \left(Z_{2k+1}^2(g)\wedge p\right)\right.\\
&\quad \left. + \left(Z_1^2(g)\wedge p\right) -\left(Z_{2n+1}^2(g)\wedge p\right)\right)=d(p)\;\;\;\pr_{\mu}-\text{a.s.}
\end{align*}
and
\begin{align*}
\lim_{n\to\infty} \frac{1}{2n+1}\sum_{k=1}^{2n+1} \left(Z_k^2(g)\wedge p\right)
&=\lim_{n\to\infty} \frac{2n}{2n+1}\cdot \frac{1}{2n}\left(\sum_{k=1}^{2n} \left(Z_{k}^2(g)\wedge p\right) + \left(Z_{2n+1}^2(g)\wedge p\right)\right)\\
&=d(p)\;\;\;\pr_{\mu}-\text{a.s.,}
\end{align*}
Obviously, condition \eqref{cnd:b1} is achieved by taking $p=\infty$.
\end{proof}

\begin{remark}\label{rem:6}
Note that $\sup_{n\in\n} \ew_{\mu}\left[|Z_n(g)|^{\zeta}\right]<\infty$. This follows in a manner analogous to the proof of Lemma \ref{lem:3}. Specifically, it suffices to observe that
$$|Z_n(g)|^{\zeta}\leq 4^{\zeta-1}\left(|\chi_g(\Phi_n)|^{\zeta}+ |\chi_g(\Phi_{n-1})|^{\zeta}\right)+2^{\zeta-1}\norma{g}_{\infty}^{\zeta}\quad\text{for all}\quad n\in\n,$$
and then use Remark \ref{rem:4}.
\end{remark}

\begin{lemma}\label{lem:7}
Condition \eqref{cnd:b2} holds.
\end{lemma}
\begin{proof}
Applying Lemma \ref{lem:6}, together with Lebesgue's dominated convergence theorem, we can conclude that
\begin{equation}
\label{e3:3}
\lim_{n\to\infty} \frac{1}{n}\sum_{k=1}^n\ew_{\mu}\left[Z_k^2(g)\wedge p\right]=\ew_{\mu_*}\left[Z_1^2(g)\wedge p\right]\quad\text{for all}\quad p\in\n\;\;\;\pr_{\mu}\text{ - a.s.}
\end{equation}
Further, exactly as in the proof of \cite[Lemma 4.4]{b:czapla-strassen}, it can be checked that
\begin{align}
\begin{split}
\label{e3:4}
\left|\ew_{\mu}\left[Z_k^2(g)\wedge p\right]-\ew_{\mu}\left[Z_k^2(g)\right] \right|&\leq 2\sup_{m\in\n}\ew_{\mu}\left[Z_m^2(g)\mathbbm{1}_{\left\{ Z_m^2(g)\geq p\right\} }\right]\\
&\leq 2p^{-(\zeta-2)/2}\sup_{m\in\n}\ew_{\mu}\left[|Z_m(g)|^{\zeta}\right] \quad\text{for all}\quad k,p\in\n,
\end{split}
\end{align}
where $\sup_{m\in\n}\ew_{\mu}\left[|Z_m(g)|^{\zeta}\right]<\infty$ by Remark \ref{rem:6}. On the other hand, we can write
\begin{align*}
\left|\frac{1}{n}\sum_{k=1}^n \ew_{\mu}\left[Z_k^2(g)\right]-\sigma_g^2\right|&\leq 
\frac{1}{n}\sum_{k=1}^n\left| \ew_{\mu}\left[Z_k^2(g)\right]-\ew_{\mu}\left[Z_k^2(g)\wedge p\right]\right|\\
&\quad+\frac{1}{n}\sum_{k=1}^n\left| \ew_{\mu}\left[Z_k^2(g)\wedge p\right]-\ew_{\mu_*}\left[Z_1^2(g)\wedge p\right]\right|\\
&\quad+\left|\ew_{\mu_*}\left[Z_1^2(g)\wedge p\right]-\ew_{\mu_*}\left[Z_1^2(g)\right] \right|
\quad\text{for all} \quad n,p\in\n.
\end{align*}
Consequently, using \eqref{e3:3} and \eqref{e3:4} we obtain
\begin{align*}
\limsup_{n\to\infty}\left|\frac{1}{n}\sum_{k=1}^n \ew_{\mu}\left[Z_k^2(g)\right]-\sigma_g^2\right|&\leq 2p^{-(\zeta-2)/2}\sup_{m\in\n}\ew_{\mu}\left[|Z_m(g)|^{\zeta}\right]\\
&\quad+\left|\ew_{\mu_*}\left[Z_1^2(g)\wedge p\right]-\ew_{\mu_*}\left[Z_1^2(g)\right] \right|\quad \text{for all}\quad p\in\n, 
\end{align*}
which, after passing with $p$ to $\infty$, leads to \eqref{cnd:b2}.
\end{proof}

\begin{lemma}\label{lem:8}
Conditions \eqref{cnd:b3} and \eqref{cnd:b4} are satisfied.
\end{lemma}
\begin{proof}
To simplify notation, we will write $s_n^2$ instead of $s_n^2(\mu,g)$. Let $\hat{\zeta}:=\zeta \wedge 4$, $\beta:=\hat{\zeta}-2$ (then $\beta\in (0,2)$), and choose $\bar{n}$ sufficiently large such that $s_n^2>0$ for all $n\geq \bar{n}$ (this can be done due to Lemma \ref{lem:7}, which entails \eqref{cnd:b5}). Then, for any $\delta,\varepsilon>0$ and every $n\geq\bar{n}$, we have
\begin{align*}
s_n^{-4}\ew_{\mu}\left[Z_n^4(g) \mathbbm{1}_{\left\{|Z_n(g)|<\delta s_n \right\}}\right]
&=s_n^{-4}\ew_{\mu}\left[|Z_n(g)|^{2+\beta}|Z_n(g)|^{2-\beta} \mathbbm{1}_{\left\{|Z_n(g)|<\delta s_n \right\}}\right]\\
&\leq \delta^{2-\beta}s_n^{-(2+\beta) }\ew_{\mu}\left[|Z_n(g)|^{2+\beta}\right]
\leq \delta^{4-\hat{\zeta}}s_n^{-\hat{\zeta}}\sup_{k\in\n}\ew_{\mu}\left[|Z_k(g)|^{\hat{\zeta}}\right]
\end{align*}
and
\begin{align*}
s_n^{-1}\ew_{\mu}\left[|Z_n(g)| \mathbbm{1}_{\left\{|Z_n(g)|\geq \varepsilon  s_n \right\}}\right]
&=s_n^{-1}\ew_{\mu}\left[\frac{|Z_n(g)|^{2+\beta}}{|Z_n(g)|^{1+\beta}}  \mathbbm{1}_{\left\{|Z_n(g)|\geq \varepsilon  s_n \right\}}\right]\\
&\leq \varepsilon^{-(1+\beta)}s_n^{-(2+\beta)}\ew_{\mu}\left[|Z_n(g)|^{2+\beta}\right]\\
&\leq \varepsilon^{1-\hat{\zeta}}s_n^{-\hat{\zeta}}\sup_{k\in\n}\ew_{\mu}\left[|Z_k(g)|^{\hat{\zeta}}\right],
\end{align*}
where $\sup_{k\in\n}\ew_{\mu}\left[|Z_k(g)|^{\hat{\zeta}}\right]<\infty$ by Remark \ref{rem:6}. Additionally, condition \eqref{cnd:b2} (established in Lemma \ref{lem:7}) enables us to select $c\in (0,\infty)$ such that $n^{\hat{\zeta}/2}s_n^{-\hat{\zeta}}\leq c$ for all $n\in\n$. Combining this with the bounds derived above, we see that the series in  \eqref{cnd:b3} and \eqref{cnd:b4} indeed converge~(for any $\delta,\varepsilon>0$).
\end{proof}

\begin{lemma}\label{lem:9}
If $\sigma_g^2>0$, then the martingale $\{M_n\}_{n\in\n}$ satisfies the LIL (under $\pr_{\mu}$) in the manner of \eqref{e:LIL_mart}.
\end{lemma}
\begin{proof}
In light of Remark \ref{rem:5}, the assertion follows directly from Lemmas \ref{lem:6}-\ref{lem:8}.
\end{proof}

\subsection{Summary}
Based on the results from the previous two subsections, the proof of the main theorem can now be outlined as follows:

\begin{proof}[Proof of Theorem \ref{thm:main}]
First of all, $\sigma_g^2<\infty$ by Lemma \ref{lem:3}. Further, Lemma \ref{lem:1} reduces the proof to verifying that the limits stated in the theorem hold along the integers. By decomposition \eqref{e:decomposition} and Lemma \ref{lem:2}, this task, in turn, reduces to establishing the LIL for $\{M_n(g)\}_{n \in \mathbb{N}}$, namely, the limits in \eqref{e:LIL_mart}. The latter, however, has been shown in Lemma~\ref{lem:9}.
\end{proof}
The remaining challenge is to prove Lemma \ref{lem:5}. This will be the focus of the next section.

\section{Proof of Lemma \ref{lem:5}}\label{sec:4}
We will prove the continuity of  $\I{i}{p}{d}$. The reasoning for $\S{i}{p}{d}$ proceeds analogously. The proof will rely on showing that, on every compact set, the function  $\I{i}{p}{d}$ can be expressed as the pointwise limit of a double sequence of Lipschitz continuous functions, sharing the same Lipschitz constant. 

Throughout the section, we assume that all hypotheses \ref{cnd:a1}–\ref{cnd:a4} are satisfied, although not all of them are used in every individual auxiliary result.

Let us begin with demonstrating that conditions \ref{cnd:a1} and \ref{cnd:a3} ensure the mean-square continuity of $t\mapsto g(\Phi_t)$. This, in turn, will enable us to approximate (in the mean-square sense) the integrals involved in the martingale increments $Z_n(g)$ by using suitable elementary processes. To establish this result, we require the following (cf. \cite[Proposition 3.6.6]{b:ziemlanska}):

\begin{lemma}\label{lem-c:1}
The map $[0,\infty)\ni t \mapsto \delta_x P_t\in (\mathcal{M}_1(X), d_{BL})$ is continuous for every $x\in X$.
\end{lemma}
\begin{proof}
Due to the semigroup property, condition \ref{cnd:a1}, in fact, guarantees that 
$$\lim_{t\to t_0^+} P_tf(x)=P_{t_0}f(x)\quad\text{for any}\quad f\in BC(X),\;x\in X,\;t_0\geq 0.\vspace{-0.2cm}$$
This clearly implies that, for every $x\in X$, the map $t \mapsto \delta_x P_t$ is right-continuous. Therefore, it remains to establish its left-continuity. To this end, note first that \ref{cnd:a3} ensures the \hbox{(uniform)} equicontinuity of the family $\{P_t f:\; t\geq 0,\; f\in BL(X),\;\lip f\leq 1\}$, since for every $t\geq 0$ and any $f\in BL(X)$ with $\lip f\leq 1$, we have
$$|P_t f(x)-P_t f(y)|=|\<f, \delta_x P_t\>-\<f, \delta_y P_t\>|\leq d_W(\delta_x P_t, \delta_y P_t)\leq \rho(x,y)\quad\text{for}\quad x,y\in X.$$
Hence, according to \cite[Theorem 4.1]{b:szarek_hille}, the family $\{(\cdot)P_t:\;t\geq 0\}$ of Markov operators from $(\mathcal{M}_1(X), d_{BL})$ into itself is point-wise equicontinuous. Now, let $x\in X$, $t_0\geq 0$, and $\varepsilon>0$. In view of the above observation, there exists~$\delta>0$ such that
\begin{equation}
\label{e-c:1}
\sup_{t\geq 0}\, d_{BL}(\delta_x P_t, \mu P_t)<\varepsilon\quad\text{for any}\quad\mu\in\mathcal{M}_1(X)\quad\text{satisfying}\quad d_{BL}(\delta_x,\mu)<\delta.
\end{equation}
On the other hand, from assumption \ref{cnd:a1} it follows that $\delta_x P_t \stackrel{w}{\to} \delta_x$ as $t\to 0^+$, which gives \hbox{$\lim_{t\to 0^+} d_{BL}(\delta_x P_t,\, \delta_x)=0$}. Thus, we can choose $s_0\in (0,t_0)$ such that $d_{BL}(\delta_x,\, \delta_x P_s)<\delta$ for all $s\in (0,s_0)$. Then, referring to \eqref{e-c:1}, we infer that
$$d_{BL}(\delta_x P_{t_0-s},\,\delta_x P_{t_0})=d_{BL}(\delta_x P_{t_0-s},\,(\delta_x P_s) P_{t_0-s})<\varepsilon\quad \text{for all}\quad s\in (0,s_0).$$
In other words, we have shown that $\lim_{t\to t_0^{-}} d_{BL}(\delta_x P_t,\,\delta_x P_{t_0})=0$, which, due to the arbitrariness of $t_0$, completes the proof.
\end{proof}

\begin{lemma}\label{lem-c:2}
The map $[0,\infty)\ni t \mapsto g(\Phi_t)\in \mathcal{L}^2(\pr_x)$ is continuous for every $x\in X$.
\end{lemma}
\begin{proof}
Fix $x\in X$ and $t\geq 0$. Since
$$|g(u)-g(v)|^2\leq \norma{g}_{BL}^2(\rho(u,v)\wedge 2)^2=\norma{g}_{BL}^2\left( \rho(u,v)^2 \wedge 4 \right)\quad\text{for any}\quad u,v\in X,$$
it follows that, for every $s\geq 0$,
\begin{align*}
\lnorma{x}{g(\Phi_s)-g(\Phi_t)}^2
&=\ew_x\left[\left|g(\Phi_s)-g(\Phi_t) \right|^2\right]
\leq \norma{g}_{BL}^2\ew_x\left[\rho(\Phi_s,\Phi_t)^2\wedge 4\right]\\
&=\norma{g}_{BL}^2\int_X \int_X \left(\rho(u,v)^2\wedge 4\right) P_{|t-s|}(u,dv)P_{s\wedge t}(x,du),
\end{align*}
where the final equality comes from \eqref{e:Ex}. Letting
$$L_u(v):=\rho(u,v)^2 \wedge 4\quad\text{for}\quad u,v\in X,$$
we therefore obtain
\begin{equation}
\label{e-c:2} \lnorma{x}{g(\Phi_s)-g(\Phi_t)}^2\leq \norma{g}_{BL}^2\int_X P_{|t-s|}L_u(u)\,P_{s\wedge t}(x,du)\quad\text{for}\quad s\geq 0.
\end{equation}
Further, note that, as $L_u\in BC(X)$ for every $u\in X$, hypothesis \ref{cnd:a1} gives
\begin{equation}
\label{e-c:3}
\lim_{s\to t} P_{|t-s|}L_u(u)=L_u(u)=0.
\end{equation}

Clearly, for every $s>t$, the integral on the right-hand of \eqref{e-c:2} is taken with respect to $P_t(x,du)$. Hence, applying the Lebesgue dominated convergence theorem, from \eqref{e-c:2} and~\eqref {e-c:3} we can conclude that
$$\lim_{s\to t^+} \lnorma{x}{g(\Phi_s)-g(\Phi_t)}=0.$$

What is now left is to prove that the limit as $s\to t^{-}$ is $0$ as well. This is more involved, since for $s<t$, the measure of integration in \eqref{e-c:2} is changing with $s$. To prove this part, let $\{s_n\}_{n\in\n}$ be an arbitrary sequence in $[0,t)$ such that $\lim_{n\to\infty} s_n=t$, and define
$$f_n(u):=P_{t-s_n} L_u(u)\quad\text{for}\quad u\in X,\; n\in\n.$$
The aim is to establish that
\begin{equation}
\label{e-c:4}
\lim_{n\to \infty} \lnorma{x}{g(\Phi_{s_n})-g(\Phi_t)}=0.
\end{equation}
Having in mind \eqref{e-c:2}, for every $n\in\n$, we can write
\begin{align*}
\lnorma{x}{g(\Phi_{s_n})-g(\Phi_t)}^2&
\leq \norma{g}_{BL}^2\<f_n,\,\delta_xP_{s_n}\>\\
&\leq  \norma{g}_{BL}^2\left(|\<f_n,\,\delta_xP_{s_n}-\delta_x P_t\>|+|\<f_n,\,\delta_xP_t\>| \right)\\
&\leq \norma{g}_{BL}^2 \left(\sup_{k\in\n}|\<f_k,\,\delta_xP_{s_n}-\delta_x P_t\>|+|\<f_n,\,\delta_xP_t\>| \right).
\end{align*}
Since, due to \eqref{e-c:3}, $\lim_{n\to\infty} f_n(u)=0$ for all $u\in X$, by the dominated convergence theorem it follows that $\lim_{n\to \infty} \<f_n,\delta_x P_t\>=0$. Hence, we only need to show that
\begin{equation}
\label{e-c:5}
\lim_{n\to\infty}\sup_{k\in\n}|\<f_k,\,\delta_xP_{s_n}-\delta_x P_t\>|=0.
\end{equation}
This can be deduced from the Rao theorem (see, e.g., \cite[Theorem 8.2.18]{b:bogachev}), provided that \hbox{$\delta_x P_{s_n}\stackrel{w}{\to}\delta_x P_t$}, and that $\{f_n:n\in\n\}$ is both uniformly bounded and pointwise equicontinuous. The former holds by Lemma \ref{lem-c:1}. Further, it is clear that the family $\{f_n:n\in\n\}$ is uniformly bounded, as $|f_n(u)|\leq \norma{L_u}_{\infty}\leq 4$ for all $u\in X$ and $n\in\n$. To demonstrate its equicontinuity, first observe that
\begin{align*}
|L_u(z)-L_v(z)|&=|L_z(u) - L_z(v)|=\left|\rho(z,u)\wedge 2+\rho(z,v)\wedge 2 \right| \cdot \left|\rho(z,u)\wedge 2-\rho(z,v)\wedge 2 \right|\\
&\leq 4|\rho(z,u)-\rho(z,v)|\leq 4\rho(u,v)\quad\text{for any}\quad z,u,v\in X,
\end{align*}
which means that the maps $u\mapsto L_u(z)$ and $L_z$, $z\in X$, are Lipschitz continuous with constant~$4$. Taking into account this observation and hypothesis \ref{cnd:a3}, for each $n\in\n$ and any $u,v\in X$, we obtain
\begin{align*}
|f_n(u)-f_n(v)|&=|P_{t-s_n} L_u(u)-P_{t-s_n} L_v(v)|\\
&\leq |P_{t-s_n} (L_u-L_v)(u)|+|P_{t-s_n} L_v(u)-P_{t-s_n} L_v(v)|\\
&= \<|L_u-L_v|,\,\delta_u P_{t-s_n}\>+|\<L_v,\,\delta_u P_{t-s_n}-\delta_v P_{t-s_n}\>|\\
&\leq 4\rho(u,v)+4\,d_W(\delta_u P_{t-s_n}, \delta_v P_{t-s_n})\leq 8\rho(u,v),
\end{align*}
which, in particular yields, that $\{f_n:n\in\n\}$ is (uniformly, and thus pointwise) equicontinuous. Consequently, \cite[Theorem 8.2.18]{b:bogachev} ensures \eqref{e-c:5}, which, in turn, leads to \eqref{e-c:4}. Since the sequence $\{s_n\}_{n\in\n}$ was chosen arbitrarily, we finally get
$$\lim_{s\to t^-} \lnorma{x}{g(\Phi_s)-g(\Phi_t)}=0,$$
which completes the proof.
\end{proof}

Below, we provide certain straightforward observations on the mean-square approximation of integrals involving $g(\Phi_t)$, which will be instrumental in proving Lemma~\ref{lem-c:6}.

\begin{lemma}\label{lem-c:3}
Suppose that $-\infty<S<T<\infty$, and that $\{\left(\s{0}{n},\s{1}{n},\ldots,\s{q_n}{n}\right)\}_{n\in\n}$ is a~sequence of partitions of $[S,T]$, i.e.,
$$
S=\s{0}{n}<\s{1}{n}<\ldots<\s{q_n}{n}=T\quad\text{for}\quad n\in\n.\vspace{-0.3cm}
$$
Further, let
$$\Delta t_n:=\max_{j\in\{1,\ldots,q_n\}}\Delta \s{j}{n}\quad \text{for}\quad n\in\n,\quad\text{with}\quad \Delta t_n(j):=\s{j}{n}-\s{j-1}{n},$$
and define the processes $\Psi:=\{\Psi_t\}_{t\in [S,T]}$ and $\xi^{(n)}:=\{\xi_t^{(n)}\}_{t\in[S,T]}$, $n\in\n$, by
\begin{equation}
\label{e:el_proc}
\Psi_t:=g(\Phi_t),\quad \xi_t^{(n)}:=\sum_{j=1}^{q_n}g\left(\Phi_{\s{j-1}{n}}\right)\mathbbm{1}_{\left[\s{j-1}{n},\,\s{j}{n}\right)}(t)\quad\text{for}\quad t\in [S,T].
\end{equation}
Then, for any $x\in X$ and $n\in\n$, 
\begin{equation}
\label{e:el_proc_conv}
\Lnorma{S,T}{x}{\xi^{(n)}-\Psi}^2\leq (T-S)\hspace{-0.3cm} \sup_{\left\{(s,t)\in [S,T]^2:\;\, |s-t|\leq\Delta t_n\right\} }\lnorma{x}{\Psi_s-\Psi_t}^2
\end{equation}
and
\begin{equation}
\label{e:el_proc_est}
\lnorma{x}{\int_S^T \xi_t^{(n)}\,dt-\int_S^T \Psi_t\,dt}^2\leq (T-S)\Lnorma{S,T}{x}{\xi^{(n)}-\Psi}^2.
\end{equation}
Moreover, if $\lim_{n\to\infty} \Delta t_n=0$, then, for every $x\in X$, the right-hand side of \eqref{e:el_proc_conv} goes to $0$, as $n\to\infty$, which, in particular, means that $\int_S^T g(\Phi_t)\,dt$ exists in the mean-square sense.
\end{lemma}
\begin{proof}
Let $x\in X$. Obviously, $\Psi,\,\xi^{(n)}\in \mathcal{L}^2_{S,T}(\lambda\otimes\pr_x)$ for all $n\in\n$, since $g$ is bounded, and the processes $\Phi$ and $\xi^{(n)}$ are progressively measurable (the latter owing to the right-continuity of its sample paths), making them jointly measurable.

Further, observe that for every $n\in\n$, we have
\begin{align*}
\Lnorma{S,T}{x}{\xi^{(n)} -\Psi }^2&=\int_S^T \ew_x\left[\left(\xi_t^{(n)} -\Psi_t  \right)^2\right]\,dt\\
&=\sum_{j=1}^{q_n} \int_{\s{j-1}{n}}^{\s{j}{n}}\lnorma{x}{\Psi_{\s{j-1}{n}} -\Psi_t }^2\,dt\\
&\leq \left(\max_{j\in\{1,\ldots,q_n\}}\sup_{t\in \left[\s{j-1}{n},\,\s{j}{n}\right]}\lnorma{x}{\Psi_{\s{j-1}{n}} -\Psi_t }^2\right)\sum_{j=1}^{q_n}\Delta\s{j}{n}\\
&\leq (T-S)\hspace{-0.3cm} \sup_{\left\{(s,t)\in [S,T]^2:\;\, |s-t|\leq \Delta t_n\right\} }\lnorma{x}{\Psi_s -\Psi_t }^2,
\end{align*}
which establishes \eqref{e:el_proc_conv}.

To show \eqref{e:el_proc_est}, it suffices to apply the integral form of Jensen's inequality and Fubini's theorem; namely, for every $n\in\n$,
\begin{align*}
\lnorma{x}{\int_S^T \xi_t^{(n)} \,dt-\int_S^T \Psi_t \,dt}^2&=\ew_x\left[\left(\int_S^T \left(\xi_t^{(n)} -\Psi_t  \right)dt\right)^2\right]\\
&=(T-S)^2\, \ew_x\left[\left(\frac{1}{T-S}\int_S^T \left(\xi_t^{(n)} -\Psi_t  \right)dt\right)^2\right]\\
&\leq (T-S)^2\, \ew_x\left[\frac{1}{T-S}\int_S^T \left(\xi_t^{(n)} -\Psi_t  \right)^2 dt\right]\\
&=(T-S)\int_S^T \ew_x\left[\left(\xi_t^{(n)} -\Psi_t  \right)^2\right]dt\\
&=(T-S)\Lnorma{S,T}{x}{\xi^{(n)} -\Psi }^2.
\end{align*}

Finally, suppose that $\lim_{n\to\infty} \Delta t_n=0$. Then the convergence of the right-hand side of~\eqref{e:el_proc_conv} to $0$ (as $n\to\infty$) follows directly from Lemma \ref{lem-c:2}. More specifically, since the function \hbox{$[S,T]\ni t\mapsto \Psi_t \in\mathcal{L}^2(\pr_x)$} is uniformly continuous, given an arbitrary $\varepsilon>0$, we can choose $\delta>0$ such that
$$\lnorma{x}{\Psi_s -\Psi_t }^2<\varepsilon\quad \text{for}\quad s,t\in [S,T]\quad \text{such that}\quad |s-t|<\delta.$$
Then, letting $n_0\in\n$ be large enough so that $\Delta t_n<\delta$ for all $n\geq n_0$, we obtain
$$\sup_{\left\{(s,t)\in [S,T]^2:\;\, |s-t|\leq \Delta t_n\right\} }\lnorma{x}{\Psi_s -\Psi_t }^2<\varepsilon\quad\text{for every}\quad n\geq n_0.$$
The proof of the lemma is now complete.
\end{proof}

To proceed with our reasoning, we need to introduce one more piece of notation. Namely, given any $n\in\n$ and $f:X^n\to\mathbb{R}$, for each $j\in\{1,\ldots,n\}$ define
$$\lip_j f:=\sup \left\{\lip f(x_1,\ldots,x_{j-1},\,\cdot\,,x_{j+1},\ldots x_n):\; x_1,\ldots,x_{j-1},x_{j+1},\ldots x_n\in X \right\}.$$
In other words, $\lip_j f<\infty$ if $f$ is Lipschitz continuous with respect to the $j$-th variable, with a Lipschitz constant independent of the fixed values of the other variables, and $\lip_j f=\infty$ otherwise.

The following result is a continuous and more subtle version of \cite[Lemma 1]{b:bolt_majewski}. It will play a crucial role in establishing the Lipschitz equicontinuity of the sequence of functions approximating $\mathcal{I}^{(i)}_{g,p,d}$ on a fixed compact set (see Lemma \ref{lem-c:8}), to be defined soon. Additionally, we will use it in the proof of the subsequent auxiliary lemma, which forms the foundation for constructing this sequence.

\begin{lemma}\label{lem-c:4}
Let $n\in\n\backslash\{1\}$, and suppose that $f\in BM(X^n)$ is such that $\lip_j f<\infty$ for every $j\in\{1,\ldots,n\}$. Further, let $0<s_1<s_2\ldots<s_n$, and define $F:X\to\mathbb{R}$ by
$$F(x):=\ew_x\left[f\left(\Phi_{s_1},\ldots,\Phi_{s_n} \right) \right]\quad\text{for any}\quad x\in X.$$
Then $$\lip F\leq \sum_{j=1}^n (\lip_j f) e^{-\gamma s_j}.$$
\end{lemma}
\begin{proof}
Put $s_0:=0$, $L_0:=0$ and $L_j:=\lip_j f$ for $j\in\{1,\ldots,n\}$. Further, for each \hbox{$k\in\{1,\ldots,n+1\}$}, define $F_k:X^k\to\mathbb{R}$ by setting 
$$F_{n+1}(x_0,\ldots,x_n):=f(x_1,\ldots,x_n),\vspace{-0.35cm}$$
and
$$F_k(x_0,\ldots,x_{k-1}):=\int_X\ldots\int_X f(x_1,\ldots,x_n)\,P_{s_n-s_{n-1}}(x_{n-1},dx_n)\ldots P_{s_k-s_{k-1}}(x_{k-1},dx_k)$$
with $s_0:=0$, if $k\leq n$. Then $F=F_1$ due to \eqref{e:Ex}, and we have
\begin{equation}
\label{e-c:6}
F_k(x_0,\ldots,x_{k-1})=\int_X F_{k+1}(x_0,\ldots,x_k)P_{s_k-s_{k-1}}(x_{k-1},dx_k)\quad\text{for}\quad k\in\{1,\ldots,n\}.
\end{equation}

We will estimate $\lip F=\lip_1 F_1$ by a recursive procedure. For this aim, first note that
\begin{equation}
\label{e-c:7}
\lip_{n+1} F_{n+1}=L_n,\quad\text{and}\quad \lip_j F_{j+1}\leq L_{j-1}\quad\text{for any}\quad j\in\{1,\ldots,n\}.
\end{equation}
To examine $\lip_j F_j$ for $j\leq n$, let $j\in\{1,\ldots,n\}$, $u,v\in X$ and, in the case where $j\geq 2$, also fix $x_0,\ldots,x_{j-2}\in X$. Then, by appealing to \eqref{e-c:6}, we can conclude that
\begin{align*}
&F_j(x_0,\ldots,x_{j-2},u)-F_j(x_0,\ldots,x_{j-2},v)\\
&=\int_X F_{j+1}(x_0,\ldots,x_{j-2},u,x_j)P_{s_j-s_{j-1}}(u,dx_j)-\int_X F_{j+1}(x_0,\ldots,x_{j-2},v,x_j)P_{s_j-s_{j-1}}(v,dx_j)\\
&=\int_X F_{j+1}(x_0,\ldots,x_{j-2},u,x_j)-F_{j+1}(x_0,\ldots,x_{j-2},v,x_j) P_{s_j-s_{j-1}}(u,dx_j)\\
&\quad+\int_X F_{j+1}(x_0,\ldots,x_{j-2},v,x_j)\left(P_{s_j-s_{j-1}}(u,dx_j)-P_{s_j-s_{j-1}}(v,dx_j) \right),
\end{align*}
where, in case of $j=1$, $F_j(x_0,\ldots,x_{j-2},u)$ and $F_{j+1}(x_0,\ldots,x_{j-2},u,x_j)$ should be interpreted as $F_1(u)$ and $F_2(u,x_1)$, respectively. From hypothesis \ref{cnd:a3} it now follows that
\begin{align*}
|F_j(x_0,\ldots,&x_{j-2},u)-F_j(x_0,\ldots,x_{j-2},v)|\\
&\leq \left(\lip_j F_{j+1}\right)\rho(u,v)+ \left(\lip_{j+1} F_{j+1} \right)d_W\left(\delta_u P_{s_j-s_{j-1}},\,\delta_v P_{s_j-s_{j-1}}\right)\\
&\leq \left(\lip_j F_{j+1}+\left(\lip_{j+1} F_{j+1}\right)e^{-\gamma(s_j-s_{j-1})}\right)\rho(u,v).
\end{align*}
Consequently, taking into account \eqref{e-c:7}, we see that
\begin{gather*}
\lip_n F_n\leq L_{n-1}+L_n e^{-\gamma(s_n-s_{n-1})},\\
\lip_j F_j\leq L_{j-1}+\left(\lip_{j+1} F_{j+1}\right) e^{-\gamma(s_j-s_{j-1})}\quad\text{for all}\quad j\in\{1,\ldots,n-1\}.
\end{gather*}
This finally leads to
\begin{align*}
\lip F&=\lip_1 F_1\leq L_0+\left(\lip_2 F_2\right) e^{-\gamma(s_1-s_0)}
\leq \left(L_1+ \left(\lip_3 F_3\right) e^{-\gamma(s_2-s_1)}\right)e^{-\gamma(s_1-s_0)}\\
&=L_1e^{-\gamma s_1}+\left(\lip_3 F_3\right)e^{-\gamma s_2}\leq\ldots\leq \sum_{k=1}^{n-2} L_k e^{-\gamma s_k} + \left(\lip_n F_n\right) e^{-\gamma s_{n-1}}\\
&\leq \sum_{k=1}^{n-2} L_k e^{-\gamma s_k} + \left(L_{n-1}+L_n e^{-\gamma(s_n-s_{n-1})}\right)e^{-\gamma s_{n-1}}=\sum_{k=1}^n L_k e^{-\gamma s_k},
\end{align*}
which ends the proof.
\end{proof}

\begin{lemma}\label{lem-c:5}
Let $K\subset X$ be a compact set, and let $T\in (0,\infty)$. Then, for every $\varepsilon>0$, there exists $\delta>0$ such that
$$\sup_{x\in K} \lnorma{x}{g(\Phi_s)-g(\Phi_t)}^2<\varepsilon\quad\text{for any}\quad s,t\in[0,T]\quad\text{satisfying}\quad |s-t|<\delta.$$
\end{lemma}
\begin{proof}
Fix $\varepsilon>0$, and let $f(u,v):=(g(u)-g(v))^2$ for $(u,v)\in X^2$. Then, we can write
$$
F_{s,t}(x):=\lnorma{x}{g(\Phi_s)-g(\Phi_t)}^2=\ew_x\left[f(\Phi_s,\,\Phi_t) \right]\quad \text{for}\quad x\in X,\;s,t\geq 0.
$$
It is easy to check that $\lip_j f\leq 4\norma{g}_{\infty}\lip g$ for $j\in\{1,2\}$. Hence, Lemma \ref{lem-c:4} implies that
\begin{equation}
\label{e-c:8}
\lip F_{s,t}\leq 4\norma{g}_{\infty}(\lip g)\left(e^{-\gamma s}+e^{-\gamma t}\right)\leq 8\norma{g}_{BL}^2+1=:L<\infty\quad \text{for all}\quad s,t\geq 0.
\end{equation}
Since $K$ is compact and $K\subset \bigcup_{x\in K} B(x,\,\varepsilon/(2L))$, there exist $x_1,\ldots,x_m\in K$ such that
$$K\subset\bigcup_{j=1}^m B\left(x_j, \frac{\varepsilon}{2L}\right).$$
On the other hand, from Lemma \ref{lem-c:2}, we know that $[0,T]\ni t \mapsto g(\Phi_t)\in\mathcal{L}^2(\pr_x)$ is uniformly continuous, and, therefore, we can choose $\delta>0$ for which
\begin{equation}
\label{e-c:8b}
\max_{j\in\{1,\ldots,m\}} F_{s,t}(x_j)<\frac{\varepsilon}{2}\quad \text{for any}\quad s,t\in[0,T]\quad\text{satisfying}\quad |s-t|<\delta.
\end{equation}
Now, if $x\in K$, then $x\in B(x_j,\varepsilon/(2L))$ for some $j\in\{1,\ldots,m\}$, and thus, given \eqref{e-c:8} and~\eqref{e-c:8b}, we can deduce, that for any $s,t\in [0,T]$ such that $|s-t|<\delta$, 
\begin{align*}
F_{s,t}(x)\leq |F_{s,t}(x)-F_{s,t}(x_j)|+F_{s,t}(x_j)<L\rho(x,x_j)+\frac{\varepsilon}{2}\leq L\cdot\frac{\varepsilon}{2L}+\frac{\varepsilon}{2}=\varepsilon,
\end{align*}
which completes the proof.
\end{proof}

Let us now fix $p\in\mathbb{N}\cup\{\infty\}$, $d\in [0,\infty)$, $i\in\{0,1\}$, and an arbitrary compact set \hbox{$K\subset X$}. From now on, for notational simplicity, we will write $\mathcal{I}$ and $Z_n$ instead of $\I{i}{p}{d}$ and $Z_n(g)$, respectively.
 
As announced at the beginning of this section, we want to construct a double sequence of Lipschitz continuous functions (with the same Lipschitz constant) that would converge to~$\mathcal{I}$ on the set $K$. To this end, for every $n\in\n$, choose $\delta_K(n)>0$ such that
\begin{equation}
\label{def:H1}
\sup_{x\in K} \lnorma{x}{g(\Phi_s)-g(\Phi_t)}^2<\frac{1}{n^4}\;\;\;\text{for any}\;\;\; s,t\in[0,n]\;\;\text{satisfying}\;\; |s-t|<\delta_K(n).
\end{equation}
Obviously, this can be done due to Lemma \ref{lem-c:5}. Further, let $\{\left(\s{0}{n},\s{1}{n},\ldots,\s{q_n}{n}\right)\}_{n\in\n}$ be a sequence of partitions of $[0,1]$
such that
\begin{equation}
\label{def:H2}
\Delta t_n:=\max_{j\in\{1,\ldots,q_n\}}\Delta \s{j}{n}<\delta_K(n)\quad\text{for all}\quad n\in\n.
\end{equation}
Finally, let
\begin{equation}
\label{def:H3}
\widehat{Z}_n:=\chi_g(\Phi_n)-\chi_g(\Phi_{n-1})+\sum_{j=1}^{q_n} g\left(\Phi_{\s{j-1}{n}+n-1}\right)\Delta\s{j}{n}\quad\text{for}\quad n\in\n,
\end{equation}
and, for any $n,N\in\n$, define
\begin{equation}
\label{def:H4}
\wH(x):=\ew_x\left[\left|\min_{k\in\{n,\ldots,n+N\}}\left(\frac{1}{k}\sum_{l=1}^k\left(\widehat{Z}_{2l+i}^2\wedge p \right) \right)-d\right|\wedge 1 \right]\quad\text{for}\quad x\in X.
\end{equation}

We are now prepared to establish the final two results, which will ultimately allow us to complete the proof of Lemma~\ref{lem:5}.

\begin{lemma}\label{lem-c:6}
Let $\wH$, $n,N\in\n$, be the functions determined by \eqref{def:H1}-\eqref{def:H4}. Then
$$\lim_{n\to\infty}\lim_{N\to\infty} \wH(x)=\i(x)\quad\text{for all}\quad x\in K.$$
\end{lemma}
\begin{proof}
Define
$$
\wi(x):=\ew_x\left[\left|\liminf_{n\to\infty}\left(\frac{1}{n}\sum_{k=1}^n\left(\widehat{Z}_{2k+i}^2   \wedge p \right) \right)-d \right|\wedge 1 \right]
\quad\text{for}\quad x\in X.
$$
It is not hard to verify (see the beginning of the proof of \cite[Lemma 4.2]{b:czapla-strassen}) that
\begin{equation}
\label{e-c:12}
\lim_{n\to\infty}\lim_{N\to\infty} \wH(x)=\wi(x)\quad\text{for all} \quad x\in X.
\end{equation}
Similarly, letting
$$
\H(x):=\ew_x\left[\left|\min_{k\in\{n,\ldots,n+N\}}\left(\frac{1}{k}\sum_{l=1}^k\left(Z_{2l+1}^2  \wedge p \right) \right)-d\right|\wedge 1 \right]\quad\text{for}\quad x\in X,
$$
we have
$$\lim_{n\to\infty}\lim_{N\to\infty} \H(x)=\i(x)\quad\text{for all} \quad x\in X.$$
It therefore follows that
\begin{equation}
\label{e-c:9}
\left|\i(x)-\wi(x)\right|=\lim_{n\to\infty}\lim_{N\to\infty} \left|\H(x)-\wH(x)\right|\quad\text{for all}\quad x\in X.
\end{equation}
Moreover, note that, because of the identity
$$\left|\min_{j\in J} \alpha_j-\beta\right|\wedge 1=\left|\min_{j\in J} \left((1+\beta)\wedge\alpha_j\right)-\beta\right|\wedge 1\quad\text{for}\quad \alpha_j\in \mathbb{R},\;j\in J,\;\beta\in\mathbb{R},$$
where $J\neq \emptyset$ is a finite set, the functions $\H$ and $\wH$ can be equivalently expressed as
\begin{gather}
\begin{split}
\label{e-c:10}
\H(x)=\ew_x\left[\left|\min_{k\in\{n,\ldots,n+N\}} \frac{1}{k}\left(k(1+d)\wedge\sum_{l=1}^k \left( Z_{2l+i}^2  \wedge p\right) \right) -d\right|\wedge 1 \right],\\
\wH(x)=\ew_x\left[\left|\min_{k\in\{n,\ldots,n+N\}} \frac{1}{k}\left(k(1+d)\wedge\sum_{l=1}^k \left( \widehat{Z}_{2l+i}^2  \wedge p\right) \right) -d\right|\wedge 1 \right].
\end{split}
\end{gather}
We will show that $\i|_{K}=\wi|_{K}$, which, due to \eqref{e-c:12}, will conclude the proof.

Let $x\in K$ and $n,N\in\n$. Keeping in mind \eqref{e-c:10} and using Birkhoff's inequality
$$
\bigl|\, a\wedge c - b\wedge c\,\bigr|\ \leq\ \bigl|\, a -b\, \bigr|\quad\text{for}\quad a,b,c\in\mathbb{R},
$$
together with its more general form
\begin{equation}
\label{e:min1}
\left|\min_{j\in J} a_j- \min_{j\in J} b_j  \right|\leq \max_{j\in I}|a_j-b_j|,
\end{equation}
valid for any finite set $J\neq\emptyset$, with $a_j,b_j\in\mathbb{R}$, $j\in J$, and the inequality
\begin{equation}
\label{e:min2}
\left|\left( c\wedge \sum_{j\in J} a_j\right)- \left( c\wedge \sum_{j\in J} b_j\right) \right|\leq \sum_{j\in J} \left|(a_j\wedge c) - (b_j\wedge c) \right|,
\end{equation}
satisfied for any finite set $J\neq\emptyset$, with $a_j,b_j\in[0,\infty)$, $j\in J$, and $c\in\mathbb{R}$, we see that
\begin{align*}
&\left|\H(x) - \wH(x)\right|\\
&\leq \ew_x\left[\max_{k\in\{n,\ldots,n+N\}}\frac{1}{k}\left|\left(k(1+d)\wedge\sum_{l=1}^k \left( Z_{2l+i}^2  \wedge p\right)\right)- \left(k(1+d)\wedge\sum_{l=1}^k \left( \widehat{Z}_{2l+i}^2  \wedge p\right) \right) \right| \right]\\
&\leq \ew_x\left[\max_{k\in\{n,\ldots,n+N\}} \frac{1}{k}\sum_{l=1}^k \left| \left(Z_{2l+i}^2  \wedge p \wedge k(1+d)\right)- \left(\widehat{Z}_{2l+i}^2  \wedge p \wedge k(1+d)\right)\right| \right]\\
&\leq \ew_x\left[\max_{k\in\{n,\ldots,n+N\}} \frac{1}{k}\sum_{l=1}^k \left| \left(Z_{2l+i}^2   \wedge k(1+d)\right)- \left(\widehat{Z}_{2l+i}^2  \wedge k(1+d)\right)\right| \right].
\end{align*}
This, together with the fact that
\begin{equation}
\label{e:min3}
\left|\alpha_1^2\wedge \beta - \alpha_2^2\wedge \beta\right|\leq 2\sqrt{\beta}|\alpha_1-\alpha_2|\quad\text{for any}\quad \alpha_1,\alpha_2\in\mathbb{R},\; \beta\geq 0,
\end{equation}
and the H\"older inequality, further gives
\begin{align}
\begin{split}
\label{e-c:11}
\left|\H(x) - \wH(x)\right|&\leq \ew_x\left[\max_{k\in\{n,\ldots,n+N\}}\frac{2\sqrt{(1+d)}}{\sqrt{k}}\sum_{l=1}^k\left| Z_{2l+i}  -\widehat{Z}_{2l+i}   \right|\right]\\
&\leq \frac{2(1+d)}{\sqrt{n}}\sum_{l=1}^{n+N} \ew_x\left| Z_{2l+i}  -\widehat{Z}_{2l+i}   \right|\\
&\leq \frac{2(1+d)}{\sqrt{n}}\sum_{l=1}^{n+N} \left(\ew_x\left[ \left(Z_{2l+i}  -\widehat{Z}_{2l+i}  \right)^2 \right]\right)^{1/2}.
\end{split}
\end{align}

Now, it remains to appropriately estimate the sum on the right-hand side of \eqref{e-c:11}. To do this, define $\hs{j}{m}:=\s{j}{m}+m-1$ for $j\in\{1,\ldots,q_m\}$, $m\in\n$, and
$$
\Psi_t  :=g(\Phi_t),\quad \xi_t^{(m)}  :=\sum_{j=1}^{q_m}g\left(\Phi_{\hs{j-1}{m}}\right)\mathbbm{1}_{\left[\hs{j-1}{m},\,\hs{j}{m}\right)}(t)\quad\text{for}\quad t\geq 0,\; m\in\n.
$$
Clearly, for every $m\in\n$, $(\hs{1}{m},\ldots,\hs{q_m}{m})$ is a partition of $[m-1,m]$, $\Delta \hat{t}_m=\Delta t_m$, and
$$\sum_{j=1}^{q_m}g\left(\Phi_{\hs{j-1}{m}}\right)\Delta \s{j}{m}=\int_{m-1}^m \xi_t^{(m)}  \,dt.$$
Hence, according to Lemma \ref{lem-c:3}, we obtain
\begin{align}
\begin{split}
\label{e-c:13}
\ew_x\left[\left( Z_m  -\widehat{Z}_m  \right)^2 \right]&=\lnorma{x}{\int_{m-1}^m \Psi_t  \,dt-\int_{m-1}^m \xi_t^{(m)}  \,dt}^2\leq \Lnorma{m-1,m}{x}{\Psi-\xi^{(m)}}^2\\
&\leq \sup_{\left\{(s,t)\in [m-1,m]^2:\;\, |s-t|\leq \Delta t_m\right\} }\lnorma{x}{\Psi_t  -\Psi_s  }^2\\
&\leq \sup_{y\in K}\;\sup_{\left\{(s,t)\in [0,m]^2:\;\, |s-t|<\delta_K(m)\right\} }\lnorma{y}{\Psi_t  -\Psi_s  }^2\leq \frac{1}{m^4},
\end{split}
\end{align}
where last two inequalities follow sequentially from \eqref{def:H2} and \eqref{def:H1}.

Finally, taking into account \eqref{e-c:11} and \eqref{e-c:13}, along with the fact that $x\in K$ and $n,N\in\n$ were arbitrary, we infer that
$$\limsup_{N\to\infty}\left|\H(x) - \wH(x)\right|\leq \frac{2(1+d)}{\sqrt{n}}\sum_{l=1}^{\infty}\frac{1}{(2l+i)^2}\quad \text{for all}\quad x\in K,\; n\in\n.$$
This, in light of \eqref{e-c:9} and the convergence of the series on the right-hand side, proves that \(\i|_{K} = \wi|_{K}\), thus completing the proof.
\end{proof}

\begin{remark}\label{rem-c:lip}
Suppose that $f\in L(X)$, and let \hbox{$\alpha,\beta\in [0,\infty)$}. Then
$\bar{f}:=\alpha\wedge \left(f^2+\beta\right)$ is Lipschitz continuous with $\lip \bar{f}\leq 2\sqrt{\alpha}\lip f$. To see this, it suffices to apply inequalities \eqref{e:min2} (with $J=\{1,2\}$, $a_2=b_2$) and \eqref{e:min3} successively. Indeed, for any $x,y\in X$, we get
$$
|\bar{f}(x)-\bar{f}(y)|\leq |\alpha\wedge f^2(x)- \alpha \wedge f^2(y)|
\leq 2\sqrt{\alpha} |f(x)-f(y)| \leq 2\sqrt{\alpha}\left(\lip f\right)\rho(x,y).
$$
\end{remark}

\begin{lemma}\label{lem-c:8}
The functions $\wH$, $n,N\in\n$, determined by \eqref{def:H1}-\eqref{def:H4}, are Lipschitz continuous with a common Lipschitz constant.
\end{lemma}
\begin{proof}
For every $k\in\n$, put
$$
m_0:=0,\quad m_k:=q_{2+i}+q_{4+i}+\ldots+q_{2k+i}+k,
\vspace{-0.1cm}
$$
and define $h_k:X^{m_k}\to\mathbb{R}$ by
\begin{align}
\label{e-c:14}
\begin{split}
&h_k\left(x_1,\ldots,x_{m_k}\right)\\
&\hspace{-0.1cm}:=\frac{1}{k}\left(k(1+d)\wedge \sum_{l=1}^k \left(\left(\chi_g(x_{m_l})-\chi_g(x_{m_{l-1}+1})+\sum_{j=1}^{q_{2l+i}}g\left(x_{m_{l-1}+j}\right)\Delta \s{j}{2l+i}\right)^2\hspace{-0.2cm}\wedge p\right)\right)\hspace{-0.1cm}.
\end{split}
\end{align}
Further, fix $n,N\in\n$, and let $f_{n,N}:X^{m_{n+N}}\to\mathbb{R}$ be given by
\begin{equation}
\label{e-c:15}
f_{n,N}\left(x_1,\ldots,x_{m_{n+N}}\right):=\left|\min_{k\in\left\{n,\ldots,n+N\right\}}h_k\left(x_1,\ldots,x_{m_k}\right)-d \right|\wedge 1.
\end{equation}
Moreover, recall that
$$\widehat{Z}_{2l+i}:=\chi_g(\Phi_{2l+i})-\chi_g(\Phi_{2l+i-1})+\sum_{j=1}^{q_{2l+i}} g\left(\Phi_{\s{j-1}{2l+i}+2l+i-1}\right)\Delta\s{j}{2l+i}\quad\text{for}\quad l\in\n.$$
Then, taking into account \eqref{e-c:10}, it is easily seen that $\wH$ can be expressed as
\begin{align}
\arraycolsep=2pt
\label{e-c:16}
\begin{split}
\wH(x)=\ew_xf_{n,N}\big(
\begin{array}[t]{ccccc}
\Phi_{2+i-1}, & \Phi_{\s{1}{2+i}+2+i-1} &\ldots & \Phi_{\s{q_{2+i}-1}{2+i}+2+i-1}, & \Phi_{2+i},\\
\Phi_{4+i-1}, &\Phi_{\s{1}{4+i}+4+i-1} & \ldots & \Phi_{\s{q_{4+i}-1}{4+i}+4+i-1},& \Phi_{4+i},\\
\vdots & \vdots &  & \vdots & \vdots\\
\Phi_{2M+i-1}, & \Phi_{\s{1}{2M+i}+2M+i-1} & \ldots &\Phi_{\s{q_{2M+i}-1}{2M+i}+2M+i-1},& \Phi_{2M+i}\;\,\big),
\end{array}
\end{split}
\end{align}
where $M:=n+N$, and the consecutive subscripts of $\Phi$ in the $l$-th row (where \hbox{$l\in\{1,\ldots,M\}$}) correspond to variables with indices $$m_{l-1}+1,\, m_{l-1}+2,\ldots, m_{l-1}+q_{2l+i}, \, m_{l-1}+q_{2l+i}+1=m_l.$$

We shall now investigate the Lipschitz continuity of the mappings  $h_k$, $k\in\n$, with respect to their individual variables. To this end, let $k\in\n$. Since each variable of $h_k$ appears in only one term of the sum over $l$, involved in \eqref{e-c:14}, it follows that the map $x\mapsto h_k(x_1,\ldots,x_{r-1},x,x_{r+1},\ldots,x_{m_k})$ (where $x_q$ are fixed for $q\neq r$) takes one of the following three forms (under the convention that $\sqrt{p}:=\infty$ if $p=\infty$):
\begin{gather*}
x\mapsto \frac{1}{k}\left(k(1+d) \wedge \left(\left(\left|-\chi_g(x)+g(x)\Delta\s{1}{2l+i}+c\right|\wedge \sqrt{p}\right)^2+\beta \right) \right)\quad\text{if}\quad r=m_{l-1}+1,\\
x\mapsto \frac{1}{k}\left(k(1+d) \wedge \left(\left(\left|g(x)\Delta\s{j}{2l+i}+c\right|\wedge \sqrt{p}\right)^2+\beta \right) \right)\quad\text{if}\quad r=m_{l-1}+j,\\
x\mapsto \frac{1}{k}\left(k(1+d) \wedge \left(\left(\left|\chi_g(x)+c\right|\wedge \sqrt{p}\right)^2+\beta \right) \right)\quad\text{if}\quad r=m_l,
\end{gather*}
where $l\in\{1,\ldots,k\}$,  $j\in\{2,\ldots,q_{2l+i}\}$, and $c\in\mathbb{R}$, $\beta\geq 0$ are certain constants depending on the fixed values of variables $x_q$, $q\neq r$. Consequently, letting 
$$L:=2(d+1)(\lip\chi_g+\lip g)$$
(which is finite by Remark \ref{rem:lip_chi}), and referring to Remark \ref{rem-c:lip}, we can conclude that, for every~$l\in\{1,\ldots,k\}$,
\begin{align*}
\lip_{m_{l-1}+1} h_k& \leq \frac{2}{k}\sqrt{k(d+1)}\left(\lip\chi_g+\Delta\s{1}{2l+i}\lip g\right)\leq 2(d+1)\left(\lip\chi_g+\Delta\s{1}{2l+i}\lip g\right)\\
&\leq L\left(1+\Delta\s{1}{2l+i} \right),
\end{align*}
and, analogously,
$$
\lip_{m_{l-1}+j} h_k\leq 2(1+d)\Delta\s{j}{2l+i}\lip g\leq L\Delta \s{j}{2l+i}\quad\text{when}\quad j\in\left\{2,\ldots,q_{2l+i} \right\},$$
$$
\lip_{m_l} h_k \leq 2(1+d)\lip\chi_g \leq L.$$

Having established the above and applying \eqref{e:min1}, which entails
$$\lip_r f_{n,N}\leq \max\left\{\lip_r h_k:\; n\leq k\leq n+N,\; m_k\geq r  \right\}  \quad\text{for}\quad r\in\{1,\ldots, m_{n+N}\},$$
we infer that, for any \hbox{$l\in\{1,\ldots, n+N\}$} and $j\in\{2,\ldots,q_{2l+i}\}$,
$$
\lip_{m_{l-1}+1} f_{n,N}\leq L\left(1+\Delta\s{1}{2l+i} \right), \;\; \lip_{m_{l-1}+j} f_{n,N}\leq L\Delta \s{j}{2l+i},\;\;\text{and}\;\;  \lip_{m_l} f_{n,N}\leq L.
$$

Finally, from \eqref{e-c:16} and Lemma \ref{lem-c:4} it follows that
\begin{align*}
\lip \wH &\leq \sum_{l=1}^{n+N} \left(\lip_{m_{l-1}+1} f_{n,N}\right)\exp\left(-\gamma(2l+i-1)\right)\\
&\quad+\sum_{l=1}^{n+N}\sum_{j=2}^{q_{2l+i}} \left(\lip_{m_{l-1}+j} f_{n,N}\right)\exp\left(-\gamma\left(\s{j-1}{2l+i}+2l+i-1 \right)\right)\\
&\quad+\sum_{l=1}^{n+N} \left(\lip_{m_l} f_{n,N}\right)\exp\left(-\gamma(2l+i)\right),
\end{align*}
which, further gives
\begin{align*}
\lip \wH &\leq L\left(\sum_{l=1}^{n+N} \left(1+\Delta\s{1}{2l+i} \right)\exp\left(-\gamma(2l+i-1)\right)\right.\\
&\quad\left.+\sum_{l=1}^{n+N}\sum_{j=2}^{q_{2l+i}} \Delta \s{j}{2l+i}\exp\left(-\gamma\left(\s{j-1}{2l+i}+2l+i-1 \right)\right)\right.\\
&\quad \left.+\sum_{l=1}^{n+N} \exp\left(-\gamma(2l+i)\right)\right)\leq L\left(A_{n,N}+B_{n,N} \right),
\end{align*}
where
\begin{gather*}
A_{n,N}:=\sum_{l=1}^{n+N} \left(2 \exp\left(-\gamma(2l+i-1)\right)+\exp\left(-\gamma(2l+i)\right) \right),
\\ B_{n,N}:= \sum_{l=1}^{n+N}\sum_{j=1}^{q_{2l+i}} \Delta \s{j}{2l+i}\exp\left(-\gamma\left(\s{j-1}{2l+i}+2l+i-1 \right)\right).
\end{gather*}
Clearly,
$$A_{n,N}\leq 3\sum_{l=1}^{\infty} \exp\left(-\gamma(2l-1)\right)=\frac{3 e^{\gamma}}{e^{2\gamma}-1}.$$
To estimate $B_{n,N}$, note first that, for any $l\in\{1,\ldots,n+N\}$ and $j\in\{1,\ldots,q_{2l+i}\}$,
\begin{align*}
-\gamma\left(\s{j-1}{2l+i}+2l+i-1 \right)&=\gamma \Delta \s{j}{2l+i}-\gamma\left(\s{j}{2l+i}+2l+i-1 \right)\\
&\leq \gamma -\gamma\left(\s{j}{2l+i}+2l+i-1 \right).
\end{align*}
This shows that, for every $l\in\{1,\ldots,n+N\}$,
\begin{align*}
\sum_{j=1}^{q_{2l+i}} \Delta \s{j}{2l+i}&\exp\left(-\gamma\left(\s{j-1}{2l+i}+2l+i-1 \right)\right)\\
&\leq e^{\gamma} \sum_{j=1}^{q_{2l+i}} \Delta \s{j}{2l+i}
\exp\left(-\gamma\left(\s{j}{2l+i}+2l+i-1 \right)\right)\leq e^{\gamma} \int_{2l+i-1}^{2l+i} e^{-\gamma t}\,dt,
\end{align*}
where the last inequality follows from the fact that the middle sum is a lower Riemann sum for the function \hbox{$[2l+i-1,\,2l+i]\ni t\mapsto e^{-\gamma t}$}.
Thus, we obtain
$$
B_{n,N}\leq e^{\gamma} \int_{2+i-1}^{2(n+N)+i} e^{-\gamma t}\,dt\leq e^{\gamma} \int_0^{\infty} e^{-\gamma t}\,dt=\frac{e^{\gamma}}{\gamma}.
$$

We have therefore shown that
$$\lip \wH\leq Le^{\gamma}\left(\frac{{3}}{e^{2\gamma}-1}+\frac{1}{\gamma}\right)\quad\text{for all}\quad n,N\in\n,$$
which ends the proof.
\end{proof}

\begin{proof}[Proof of Lemma \ref{lem:5}]
Let $K\subset X$ be an arbitrary compact set, and let $\{\wH:\; n,N\in\n\}$ stand for the family of functions specified by \eqref{def:H1}-\eqref{def:H4}. Then from Lemmas \ref{lem-c:6} and \ref{lem-c:8}, it follows that
$$\lip\,\i|_{K}\leq \sup_{(n,N)\in\n^2} \lip \wH<\infty.$$
This shows that $\mathcal{I}$ is (Lipschitz) continuous on every compact subset of $X$, which clearly implies the continuity of $\mathcal{I}=\I{i}{p}{d}$. As mentioned earlier, the continuity of  $\mathcal{S}^{(i)}_{g,p,d}$ can be established through reasoning similar to that presented in this section.
\end{proof}

\section{On representations of $\sigma_g^2$}\label{sec:5}
Let us now return to the issue of representing the asymptotic variance of $\{I_t(g)\}_{t\geq 0}$, mentioned at the end of Section \ref{sec:2} (just before Theorem \ref{thm:main}). In \cite[Theorem 5.1]{b:czapla-CLT-cont}, we have demonstrated that~$\sigma_g^2$, given by~\eqref{def:sigma} (or, more concisely, by \eqref{e:sigma_short}), can be expressed as
$\sigma_g^2=2\<\chi_g g,\mu_*\>$. The proof of that result relies directly on condition \ref{cnd:a1}, while the remaining assumptions serve merely to ensure the existence of a unique invariant distribution. Within the framework considered here, this identity thus holds under hypotheses \ref{cnd:a1}-\ref{cnd:a4}. Consequently, it remains to verify the second equality in \eqref{e:sigma_eq}.

Before presenting the proof, let us first state a remark regarding the so-called weak-* mean ergodicity of the semigroup under consideration, which will play an essential role in the argument.
\begin{remark}\label{rem:weak_erg}
Under hypotheses \ref{cnd:a2}-\ref{cnd:a4} the semigroup $\{P_t\}_{t\geq 0}$ is weak-* mean ergodic, that is,
$$\frac{1}{t}\int_0^t \nu P_s\,ds\stackrel{w}{\to} \mu_*, \quad\text{as}\quad t\to\infty,\quad\text{for all}\quad \nu\in\mathcal{M}_1(X).$$
Indeed, conditions \ref{cnd:a2} and \ref{cnd:a3} guarantee that $\{P_t\}_{t\geq 0}$ is Feller, and that the family \hbox{$\{P_t f:\; t\geq 0\}$} is equicontinuous for every $f\in L(X)$. Hence, in view of \cite[Corollary 5.3]{b:szarek_worm}, it suffices to show that there exists $z\in X$ such that, for every $\varepsilon>0$ and some $\alpha(\varepsilon)>0$,
\begin{equation}
\label{e:we}
\limsup_{t\to\infty} \frac{1}{t}\int_0^t \delta_x P_s(B(z,\varepsilon))\,ds\geq \alpha(\varepsilon)\quad\text{for all}\quad x\in X.
\end{equation}
This, however, follows from the asymptotic stability of $\{P_t\}_{t\geq 0}$ (established in Proposition~\ref{prop:1}). To see this, take $z\in \operatorname{supp}\mu_*$, fix an arbitrary $\varepsilon>0$, and define $\alpha(\varepsilon):=\mu_*(B(z,\varepsilon))/2>0$. Then, by the portmanteau theorem (\cite[Theorem 2.1]{b:billingsley}), we get
$$\liminf_{s\to\infty} \delta_x P_s (B(z,\varepsilon))\geq \mu_*(B(z,\varepsilon))>\alpha(\varepsilon)\quad\text{for all}\quad x\in X.
$$
This yields that, for any given $x\in X$, there exists $t_x> 0$ such that \hbox{$\delta_x P_s (B(z,\varepsilon))>\alpha(\varepsilon)$} for all $t\geq t_x$. Consequently,
$$\frac{1}{t}\int_0^t \delta_x P_s (B(z,\varepsilon))\,ds\geq \frac{1}{t}\int_{t_x}^t \delta_x P_s (B(z,\varepsilon))\,ds\geq \left(1-\frac{t_x}{t}\right)\alpha(\varepsilon) \quad \text{for}\quad t\geq t_x,$$
which, in particular, implies \eqref{e:we}.
\end{remark}

\begin{proposition}
Suppose that conditions \ref{cnd:a2}-\ref{cnd:a4} are satisifed. Then 
\begin{equation}
\label{def2:sigma}
\lim_{t\to \infty}\frac{\ew_{\mu}\left[I_t^2(g) \right]}{t}=2\<g\chi_g,\mu_*\>,
\end{equation}
with $I_t(g)$ defined by \eqref{def:I}.
\begin{proof}
First of all, note that  $g\chi_g$ is continuous (as $\chi_g\in L(X)$ due to Remark \ref{rem:lip_chi}) and that, by Remark \ref{rem:4},
$$\sup_{t\ge 0}\<|g\chi_g|^2,\,\mu P_t\>\leq \norma{g}_{\infty}^2 \sup_{t\geq 0} \ew_{\mu}\left[\chi_g^2(\Phi_t) \right]<\infty.$$
Accordingly, in light of \cite[Lemma 8.4.3]{b:bogachev}) (cf. also \cite[Lemma 2.1]{b:czapla-CLT-cont}), the weak-* mean ergodicity of the semigroup $\{P_t\}_{t\geq 0}$ (see Remark \ref{rem:weak_erg}) implies that
$$\lim_{t\to\infty}\frac{1}{t}\int_0^t\<g\chi_g,\,\mu P_u\>du=\<g\chi_g,\,\mu_*\>.$$
Hence, it suffices to prove that
$$\lim_{t\to\infty} \left|\frac{2}{t}\int_0^t\<g\chi_g,\,\mu P_u\>du-\frac{1}{t}\ew_{\mu}\left[I_t^2(g)\right]\right|=0.$$

Let us define
$$\chi_g^{(t)}(x):=\int_0^t P_s g(x)\,ds\quad\text{for}\quad x\in X,\; t\geq 0.$$
Then, keeping in mind \eqref{e:int_square}, we can conclude that, for any $x\in X$ and $t\geq 0$,
\begin{align*}
\ew_x\left[I_t^2(g)\right]&= 2\int_0^t \int_0^s P_u(g\cdot P_{s-u}g)(x)\,du\,ds 
=2\int_0^t \int_u^t P_u(g\cdot P_{s-u}g)(x)\,ds\,du\\
&=2\int_0^t \int_0^{t-u} P_u(g\cdot P_v g)(x)\,dv\,du=2\int_0^t P_u\left(g\cdot \chi_g^{(t-u)}\right)(x)\,du,
\end{align*}
which further gives
\begin{align*}
\ew_{\mu}\left[I_t^2(g)\right]&=\int_X \ew_x\left[I_t^2(g)\right]\mu(dx)=2\int_0^t \<g \chi_g^{(t-u)} ,\, \mu P_u\>du.
\end{align*}
For every $t>0$, we can therefore write
\begin{equation}
\label{e-s1}
\left|\frac{2}{t}\int_0^t\<g\chi_g,\,\mu P_u\>du-\frac{1}{t}\ew_{\mu}\left[I_t^2(g)\right]\right|\leq \frac{2\norma{g}_{\infty}}{t}\int_0^t \<\left|\chi_g- \chi_g^{(t-u)}\right|,\,\mu P_u \> du.
\end{equation}

Now, fix an arbitrary $\varepsilon>0$. From \eqref{e:P_tgx}, it follows that
\begin{align*}
\left|\chi_g(x)- \chi_g^{(s)}(x)\right|&\leq \int_s^{\infty} |P_u g(x)|\,du
\leq C(\lip g)\left(\int_s^{\infty} e^{-\gamma u}\,du \right)(V(x)+1)\\
&=D e^{-\gamma s}(V(x)+1)\quad\text{for all}\quad x\in X,\;s\geq 0,
\end{align*}
with $D:=(C\lip g)/\gamma$. Consequently,
\begin{equation}
\label{e-s2}
\<\left|\chi_g- \chi_g^{(s)}\right|,\,\mu P_u  \>\leq De^{-\gamma s}\left(\<V,\mu P_u\>+1\right)\leq De^{-\gamma s}(c_1(\mu)+1)\quad \text{for all}\quad s,u\geq 0,
\end{equation}
and $c_1(\mu)<\infty$ by condition \ref{cnd:a4} (see Remark \ref{rem:2}). Thus we can choose $T>0$ so large that
\begin{equation}
\label{e-s3}
\<\left|\chi_g- \chi_g^{(s)}\right|,\,\mu P_u  \><\frac{\varepsilon}{4(\norma{g}_{\infty}+1)}\quad\text{for all}\quad s\geq T,\; u\geq 0.
\end{equation}
Finally, letting
$$t_0:=T\max\left\{1,\;\frac{4(\norma{g}_{\infty}+1)}{\varepsilon}D(c_1(\mu)+1)\right\},$$
and taking into account \eqref{e-s1}, \eqref{e-s2}, and \eqref{e-s3}, we infer that, for every $t>t_0$,
\begin{align*}
&\left|\frac{2}{t}\int_0^t\<g\chi_g,\,\mu P_u\>du-\frac{1}{t}\ew_{\mu}\left[I_t^2(g)\right]\right|\\
&\leq \frac{2\norma{g}_{\infty}}{t}\left(\int_0^{t-T} \<\left|\chi_g- \chi_g^{(t-u)}\right|,\,\mu P_u \>du +\int_{t-T}^t \<\left|\chi_g- \chi_g^{(t-u)}\right|,\,\mu P_u \> du.\right)\\
&\leq \frac{2\norma{g}_{\infty}}{t}\left(\frac{\varepsilon(t-T)}{4(\norma{g}_{\infty}+1)} +TD(c_1(\mu)+1)\right)\\
&=2\norma{g}_{\infty}\left(\frac{\varepsilon}{4(\norma{g}_{\infty}+1)}\left(1-\frac{T}{t}\right)+TD(c_1(\mu)+1)\frac{1}{t}\right)<\frac{\norma{g}_{\infty}\varepsilon}{\norma{g}_{\infty}+1}<\varepsilon.
\end{align*}
Since $\varepsilon$ was arbitrary, the proof is now complete.
\end{proof}
\end{proposition}

\section{An Application to an SDE with Dissipative Drift and Additive Noise}\label{sec:6}

As announced earlier, we finalize the paper with a brief note concerning the application of Theorem \ref{thm:main} to the model studied in \cite[§6.1]{b:komorowski}.

Let $(X,\<\cdot|\cdot\>)$ be a separable Hilbert space with an orthonormal basis \hbox{$\{e_n:\,n\in\n\}$}, and let $\mathcal{L}(X)$ denote the space of bounded linear operators on $X$, endowed with the operator norm. 

Consider a strongly continuous, analytic semigroup of operators  $\{S_t\}_{t\geq 0}\subset \mathcal{L}(X)$ with the infinitesimal generator $A:D(A)\to X$. Analyticity here means that $\{S_t\}_{t\geq 0}$ extends holomorphically in time to a sector of the complex plane. More precisely, there exist an open sector \hbox{$\Sigma:=\{z\in\mathbb{C}\backslash\{0\}:\;|\arg z|<\theta\}$}, with a fixed $\theta\in (0,\pi]$, and a family \hbox{$\{\bar{S}_z\}_{z\in \cl \Sigma}\subset \mathcal{L}(X)$} such that the mapping $\Sigma \ni z \mapsto \bar{S}_z \in\mathcal{L}(X)$ is analytic, $\bar{S}_t=S_t$ for each real $t\geq 0$, $\bar{S}_{w+z}=\bar{S}_w\circ\bar{S}_z$ for any $w,z \in\cl \Sigma$ satisfying $w+z\in \cl \Sigma$, and for every $x\in X$ the function $z\mapsto \bar{S}_z(x)$ is continuous at $z=0$.

Further, let \hbox{$\mathbb{W}:=\{\mathbb{W}_t\}_{t\geq 0}$} stand for an $X$-valued Wiener process of the form 
$$\mathbb{W}_t:=\sum_{n=1}^{\infty} \beta_n W_t^{(n)}e_n,\quad t\geq 0,$$
where $\{W_t^{(n)}\}_{t\geq 0}$, $n\in\n$, are independent, standard, one-dimensional Wiener processes adapted to a common filtration $\{\mathcal{G}_t\}_{t\geq 0}$, and $\{\beta_n\}_{n\in\n}$ is a square-summable sequence of real numbers. It can be easily shown that the convariance operator of $\mathbb{W}$, i.e., the unique bounded linear operator $Q:X\to X$ satisfying 
$$\ew[\<\mathbb{W}_s\,|\,x\>\<\mathbb{W}_t\,|\,y\>]=(s\wedge t)\<Qx\,|\,y
\>\quad\text{for any}\quad s,t\geq 0, \;x,y\in X,$$
is positive, self-adjoint, trace-class, and given explicitly by
$$Qx=\sum_{n=1}^{\infty} \beta_n^2\<x\,|\,e_n\>e_n\quad\text{for}\quad x\in X.$$ 

Additionally, define $\{\mathbb{W}^A_t\}_{t\geq 0}$ to be the so-called \emph{convolution process} (\cite[§5.1.2]{b:prato}), i.e.,
\vspace{-0.19cm}
$$\mathbb{W}^A_t:=\int_0^t S_{t-s}\,d\mathbb{W}_s\quad\text{for}\quad t\geq 0,$$
where the integral on the right-hand side can be interpreted as $\sum_{n=1}^{\infty}\int_0^t S_{t-s}(\beta_n e_n)\,dW_s^{(n)}$, understood in the $\mathcal{L}^2$ sense. It is well-known that $\{\mathbb{W}^A_t\}_{t\geq 0}$, defined in this way, is a Gaussian process with $X$-continuous sample paths. 

Finally, suppose we are given an arbitrary Lipschitz continuous function $F:X\to X$.

Within the above framework, one may investigate the following initial-value problem for an Itô SDE with dissipative drift and additive noise:
\begin{equation}\label{e:3}
d\Psi_t = \left(A(\Psi_t)+F(\Psi_t)\right)dt+d\mathbb{W}_t,\quad \Psi_0\sim \mu,\quad\text{where}\quad \mu\in \mathcal{M}_1(X).
\end{equation}
An $X$-valued, $\{\mathcal{G}_t\}_{t\geq 0}$-adapted process $\Psi=\{\Psi_t\}_{t\geq 0}$ with continuous sample paths and initial distribution $\mu$ is called a \emph{mild solution} to this problem (see \cite[p. 81]{b:prato}) whenever
$$
\Psi_t=S_t(\Psi_0)+\int_0^t S_{t-s} F(\Psi_s)\,ds+\mathbb{W}^A_t\quad\text{for}\quad t\geq 0,\;\;\text{a.s.}
$$

Following \cite{b:komorowski}, we now introduce the following conditions:

\begin{enumerate}[label=\textnormal{(C\arabic*)}, leftmargin=*]
\item \label{cnd:c1} There exists $\gamma_1\in\mathbb{R}$ such that $\norma{e^{\gamma_1 t} S_t}\leq 1$ for every $t\geq 0$;
\item \label{cnd:c2} There exists $\gamma_2\in\mathbb{R}$ such that 
$$\<F(x)-F(y)\,|\,x-y\>\leq -\gamma_2\norma{x-y}^2\quad\text{for any}\quad x,y\in X;$$
\item \label{cnd:c3} $\displaystyle\sup_{t\geq 0} \int_0^t \operatorname{tr}S_s^*QS_s\,ds<\infty,$
\end{enumerate}
where $\operatorname{tr} T:=\sum_{n=1}^{\infty}\<Te_n\,|\,e_n\>$ for any trace-class operator $T\in\mathcal{L}(X)$. In particular, the last hypothesis implies that $\sup_{t\geq 0} \ew[\|\mathbb{W}^A_t\|^2]<\infty$.

From \hbox{\cite[Theorem 5.5.11]{b:prato}} it follows that, if \ref{cnd:c1}-\ref{cnd:c3} hold with $\gamma_1+\gamma_2>0$, then for each $\mu\in\mathcal{M}_1(X)$ there exists a unique mild solution $\{\Psi_t^{(\mu)}\}_{t\geq 0}$ to \eqref{e:3}. It is also known that these solutions form a family of Markov processes, which corresponds to a Markov-Feller semigroup $\{H_t\}_{t\geq 0}$, so that 
$$\pr\left(\Psi_t^{(\mu)}\in A\right)=\mu H_t(A) \quad\text{for any}\quad \mu\in\mathcal{M}_1(X),\;A\in\mathcal{B}(X),\;t\geq 0.$$
Furthermore, under the same assumptions, it is shown in \cite[§6.1]{b:komorowski} that $\{H_t\}_{t\geq 0}$ satisfies hypotheses \ref{cnd:a1}-\ref{cnd:a4} with $\gamma:=\gamma_1+\gamma_2$ and any measure $\mu\in\mathcal{M}_{1,\zeta}(X)$, where $\zeta>2$. This, combined with Theorem \ref{thm:main}, finally leads us to the following conclusion:

\begin{proposition}
Suppose that hypotheses \ref{cnd:c1}-\ref{cnd:c3} hold with $\gamma_1,\gamma_2$ satisfying $\gamma_1+\gamma_2>0$, and let $\zeta>2$. Then $\{H_t\}_{t\geq 0}$ admits a unique invariant probability measure $\mu_*$, which belongs to $\mathcal{M}_{1,\zeta}(X)$. Moreover, for any $g\in BL(X)$ with \hbox{$\<g,\mu_*\>=0$},
$$\sigma_g^2:=\lim_{t\to\infty} \frac{1}{t}\ew\left[\left(\int_0^t g\left(\Psi_s^{(\mu)}\right)ds\right)^2\right]$$
is well-defined and finite, and whenever it is positive, the LIL holds for the process $$\left\{\int_0^t g\left(\Psi_s^{(\mu)}\right)\,ds\right\}_{t\geq 0},$$ with $\sigma_g^2$ serving as its asymptotic variance.
\end{proposition}

\section*{Acknowledgements}
The work of D.C. and H.W.-Ś. was supported by the University of Silesia in Katowice through the programme ``Freedom of Research | in the City of Science'' (a call within the “Research Excellence Initiative”), an important part of the strategic activities carried out under the framework of the European City of Science Katowice 2024. Part of this article was prepared at the Faculty of Physics and Applied Mathematics of the Gdańsk University of Technology, whose hospitality is gratefully acknowledged.

\section*{Competing interests}
The authors have nothing to declare.

\section*{Funding}
This research did not receive any specific grant from funding agencies in the public, commercial, or not-for-profit sectors.

\section*{Declarations of interest}
None.

\bibliographystyle{elsarticle-harv}
\footnotesize
\bibliography{ReferencesDatabase}
\end{document}